\documentclass[hidelinks,onefignum,onetabnum]{siamart220329}
\usepackage{amssymb}
\usepackage{stmaryrd}
\usepackage{physics}
\usepackage{braket}
\usepackage{booktabs}
\usepackage{algpseudocode}
\usepackage{arydshln}
\usepackage{multirow}
\usepackage{mathrsfs}
\graphicspath{{Fig/}}
\newsiamremark{assumption}{Assumption}
\newsiamremark{example}{Example}
\newsiamremark{remark}{Remark}
\newsiamthm{property}{Property}
\newsiamthm{problem}{Problem}
\newsiamthm{scheme}{Scheme}
\crefname{assumption}{assumption}{assumptions}
\Crefname{assumption}{Assumption}{Assumptions}
\crefname{proposition}{proposition}{propositions}
\Crefname{proposition}{Proposition}{Propositions}
\crefname{remark}{remark}{remarks}
\Crefname{remark}{Remark}{Remarks}
\newcommand{\Rcal}{\mathcal R}
\newcommand{\Hmu}{\mathcal H}
\newcommand{\Mmu}{\mathcal M}
\newcommand{\PIop}{\mathscr P}
\newcommand{\Playop}{\mathscr F}
\newcommand{\Gapop}{\mathscr B}
\newcommand{\Aggop}{\mathscr G}
\headers{Numerical analysis with play-type PI hysteresis}{S. Xu and L. Cao}
\title{Numerical analysis of parabolic equations with Prandtl--Ishlinskii hysteresis of play type\thanks{Submitted to the editors. \funding{This work was funded by the National Natural Science Foundation of China (grant 12371437) and the Beijing Natural Science Foundation (grant Z240001).}}}
\author{Shu Xu\thanks{School of Mathematical Sciences, Peking University, Beijing 100871, China (\email{mathxushu@pku.edu.cn}).}
\and Liqun Cao\thanks{LSEC, NCMIS, Institute of Computational Mathematics and Scientific/Engineering Computing, Academy of Mathematics and Systems Science, Chinese Academy of Sciences, Beijing 100190, China (\email{clq@lsec.cc.ac.cn}).}}
\ifpdf
\hypersetup{pdftitle={Numerical analysis of parabolic equations with Prandtl--Ishlinskii hysteresis of play type},pdfauthor={Shu Xu and Liqun Cao}}
\fi
\begin{document}
\maketitle
\begin{abstract}
Rigorous error analysis for numerical approximations of parabolic equations with hysteresis remains limited, even for the widely used Prandtl--Ishlinskii model of play type. 
In this work, we establish an \(O(h+\tau)\) error bound for an implicit Euler \(P_1\) finite element discretization. 
The analysis requires neither higher-order temporal regularity nor additional spatial regularity of the hysteresis variables.
For the temporal discretization, we exploit a convex subgradient-flow structure in a weighted Hilbert space and estimates of residual monotone terms to obtain first-order convergence.
For the spatial discretization, only the diffusive field is restricted to the finite element space, and a constraint-preserving comparison yields an \(O(h)\) semidiscrete estimate. 
The analysis is developed for play-type Prandtl--Ishlinskii operators formulated directly on a spatial Hilbert space, encompassing the canonical pointwise extension as well as more general spatially structured constraints.
\end{abstract}
\section{Introduction}
\label{sec:intro}

  In this work, we consider the numerical analysis of parabolic equations with hysteresis of the form
  \begin{equation}
    \partial_t\bigl(u+\PIop[u]\bigr)-\Delta u=f, \label{eq:pde_strong}
  \end{equation}
  where \(\PIop\) is a Prandtl--Ishlinskii (PI) hysteresis operator of play type.  The PI operator combines a linear response with a weighted superposition of play operators $\set{\Playop_r}$:
  \begin{equation}
    \PIop[u]=a u+\int_{\Rcal} c(r)\Playop_r\left[u\right]\,d\mu(r),
  \end{equation}
  with the dependence on the constitutive data and initial memories suppressed in the operator notation. Here the play operators \[\Playop_r\colon W^{1,1}(0,T;H)\rightarrow W^{1,1}(0,T;H)\] are formulated directly on the spatial Hilbert space \(H\).
  This formulation allows the hysteresis operator to encode spatial structure beyond pointwise extensions of scalar play laws.

  Prandtl--Ishlinskii models of play type and their variants have been widely used in precision actuation and hysteresis compensation \cite{al_janaideh_analytical_2011,gu_modeling_2014,zhang_prandtlishlinskii_2023,alJanaidehFundamentalsPI2023}.
  Parabolic equations with hysteresis in the principal part, as in \eqref{eq:pde_strong}, arise in electromagnetic loss calculations \cite{vankeerComputationalMethodsEvaluation1998,bermudezElectromagneticComputationsPreisach2017} and in models of flow through porous media \cite{schweizer_hysteresis_2017,gavioli_degenerate_2023}. 
  A broad related analytical theory has been developed, including results on well-posedness, regularity and homogenization  \cite{visintinDifferentialModelsHysteresis1994,francuHomogenizationHeat2003,eleuteri_wellposedness_2008,gavioli_degenerate_2025}. 
  On the numerical side, implicit-Euler finite element approximations were already considered in the early work \cite{verdiNumericalApproximationHysteresis1985,verdiNumericalApproximationPreisach1989}, where solvability, stability, and convergence were established, but explicit temporal and spatial convergence rates were not derived. 
  Quantitative error analysis under the natural regularity of hysteretic evolutions remains less developed.

  A principal difficulty in the error analysis stems from the history dependence of the PI operator.
  Unlike the instantaneous constitutive mappings considered in
  \cite{epperson_finite_1984,elliott_error_1987,verdi_numerical_1994}, the hysteresis output at a given time \(t\) depends not only on the current input but also on the input history up to \(t\).
  Moreover, this dependence differs from the fixed convolution structure considered in \cite{hornung_diffusion_1990,peszynska_finite_1996}.
  Consequently, the finite element arguments developed for these two classes of problems do not transfer directly to \eqref{eq:pde_strong}.
  However, the current family of play states retains all past information relevant to the subsequent evolution, allowing the history dependence to be encoded in an enlarged state space.
  This viewpoint underlies classical \(m\)-accretive formulations \cite{visintinDifferentialModelsHysteresis1994,little_semilinear_1994}, and we adopt it to formulate \eqref{eq:pde_strong} as a local-in-time evolution in an enlarged Hilbert space.
  The resulting formulation removes the explicit history dependence at the state level, while the limited temporal and spatial regularity of the hysteresis variables remains a central difficulty in the error analysis.

  The limited temporal regularity is already apparent for a single play operator (see Figure~\ref{fig:linear_play}). 
  Even for a smooth input, transitions between sticking and motion can make the time derivative of the play output discontinuous, and PI superposition need not smooth out this loss of regularity.
  Consequently, the hysteresis variables need not have the second-order temporal regularity required by standard consistency arguments for first-order convergence of the implicit Euler method.
  A more broadly applicable alternative is provided by the general error theory for \(m\)-accretive evolutions. Nochetto and Savaré \cite{nochettoSavareNonlinearEvolution2006} established generic \(O(\tau^{1/2})\)-type error bounds for the implicit Euler scheme in this framework and included parabolic equations with hysteresis among their applications. 
  However, these bounds do not capture the first-order temporal convergence observed in numerical studies of related hysteresis problems \cite{verdiNumericalApproximationHysteresis1985,bermudezElectromagneticComputationsPreisach2017}.

  This gap points to additional variational structure in the present hysteretic evolution. Rulla \cite{rulla_error_1996} established first-order convergence for implicit approximations of subgradient evolutions, while Nochetto et al.~\cite{nochetto_error_1998,nochetto_posteriori_2000} developed a posteriori error estimates that exploit the underlying convex potential.
  These results suggest seeking a convex subgradient-flow structure for the augmented hysteretic evolution in a suitable Hilbert space. 
  Such a structure provides a route to first-order temporal convergence through estimates of residual monotone terms, without requiring higher-order temporal regularity of the hysteresis variables.

  The spatial discretization raises a different difficulty.
  The diffusive field benefits from elliptic regularity, whereas comparable spatial regularity is generally unavailable for the hysteresis variables \cite{visintinDifferentialModelsHysteresis1994}. 
  A componentwise finite element comparison would therefore introduce interpolation errors for these variables, with no guaranteed algebraic convergence rate.
  This dependence on componentwise spatial regularity is explicit in Bartels' fully discrete error estimate \cite{bartels_quasi-optimal_2014}, whose application to kinematic-hardening elastoplasticity contains an interpolation error for the full state rate.
  Alberty and Carstensen \cite{alberty_numerical_2000}
  and Valdman \cite{valdman_mathematical_2002}
  exploit the coupled variational structure and element averaging to eliminate separate spatial interpolation terms for the internal-variable rates from their principal error estimates.
  Although their elastoplastic systems differ from the parabolic hysteresis problem considered here, their analyses suggest preserving the coupled state structure in the spatial comparison.
  Accordingly, we restrict only the diffusive field to the finite element space and preserve the hysteretic constraint in the comparison construction.

  The main contributions are summarized as follows.
  We first formulate play hysteresis directly in the spatial Hilbert space, thereby including the canonical pointwise extension and allowing more general spatially structured play constraints. We then recast the parabolic hysteresis problem as a convex subgradient flow in a weighted Hilbert space, providing the framework for the subsequent error analysis.
  For the time discretization, we prove first-order convergence of the implicit Euler scheme for forcing of bounded variation and initial data in the domain of the subdifferential, without requiring higher-order temporal regularity of the hysteresis variables. For the spatial discretization, we restrict only the diffusive field to the finite element space and construct a constraint-preserving comparison state, which yields an \(O(h)\) semidiscrete error estimate without requiring spatial interpolation estimates for the hysteresis variables. The temporal and spatial estimates together yield the space--time bound \(O(h+\tau)\) under exact spatial integration and constitutive evaluation.

  The remainder of the paper is organized as follows. Section~\ref{sec:continuous_pi_model} introduces the Hilbert-space formulation of the play-type Prandtl--Ishlinskii operator and states the corresponding parabolic hysteresis problem. Section~\ref{sec:continuous_pi_structure} develops the weighted Hilbert-space subgradient formulation and establishes the well-posedness and regularity properties used in the subsequent analysis. Section~\ref{sec:continuous_pi_time} studies the implicit Euler discretization and derives both a posteriori and a priori temporal error estimates. Section~\ref{sec:continuous_pi_space} addresses the spatial finite element approximation, including the restricted subgradient formulation, semidiscrete error estimates, and the space--time error bound. Section~\ref{sec:conclusion} concludes the paper and discusses directions for further work.


\section{Hilbert-space formulation of spatially structured Prandtl--Ishlinskii hysteresis}
  \label{sec:continuous_pi_model}

  In this section, we formulate spatially structured Prandtl--Ishlinskii hysteresis of play type through a parameterized family of Hilbert-space play operators. We then introduce the corresponding memory space and formulate the
  parabolic problem in weak form, providing the setting for the subgradient-flow formulation in Section~\ref{sec:continuous_pi_structure}.

\subsection{Vector-valued play operators and Prandtl--Ishlinskii superposition}
Fix a final time $T>0$. Let $H$ be a real separable Hilbert space endowed with a scalar product $(\cdot,\cdot)_H$.
For a closed convex set $C\subset H$, we denote its outward normal cone at a point $x\in C$ by
\[
 N_C(x):=\{\xi\in H:(\xi,y-x)_H\leq0\ \forall y\in C\}.
\]
Assume that \(0\in C\) and let \(q_0\in C\).
By \cite[Theorem~I.1.9 and Proposition~I.3.9]{krejčí1996hysteresis}, for every \(p\in[1,\infty]\) and every \(u\in W^{1,p}(0,T;H)\), there exists a unique $\xi \in W^{1,p}(0,T;H)$ satisfying the initial condition  $u(0) - \xi(0) = q_0$ and
\begin{equation}
  u(t)-\xi(t) \in C,\quad  \partial_t \xi(t) \in N_C\left(u(t)-\xi(t)\right)
  \quad\text{a.e. in }(0,T).
\end{equation}
This solution defines the vector-valued play operator with characteristic $C$:
\begin{equation}
  \Playop_C\colon   W^{1,p}(0,T;H) \times C \rightarrow W^{1,p}(0,T;H),\quad (u,q_0) \mapsto \xi.
\end{equation}
A single play element provides a simple and fundamental model of hysteresis, but its descriptive capability is limited.
Following \cite[Example~I.1.16]{krejčí1996hysteresis}, we define Prandtl--Ishlinskii models of play type by taking weighted
superpositions of play elements with different characteristics.

\begin{figure}[h]
  \centering
  \input{Fig/pi_hysteresis_figure}
  \caption{Single-play and Prandtl--Ishlinskii responses under the same input. Panel~(b) uses the threshold $r=2$, while panel~(c) shows the normalized uniform superposition $\PIop[u]=\frac12\int_0^2\Playop_r[u]\,\mathrm{d}r$. The red point marks the initial state and the arrows indicate the direction of traversal.}
  \label{fig:linear_play}
\end{figure}

\begin{definition}[Prandtl--Ishlinskii model of play type]
\label{def:continuous_pi}
Let $(\Rcal,\mathcal A,\mu)$ be a measure space and \(K\colon \Rcal \rightrightarrows H, \, r\mapsto K_r\) a set-valued map, where $K_r\subset H$ is closed and convex with $0\in K_r$ for each $r \in \Rcal$. Suppose that $c\colon \Rcal\to (0,+\infty)$ is measurable and $a\geq0$.
We call
\[
 \mathfrak D:=\bigl(\Rcal,\mathcal A,\mu,a,c,K\bigr)
\]
the constitutive data.

Given an input $u\in W^{1,1}(0,T;H)$ and a strongly measurable family of initial memories
\[
 w^0:\Rcal\to H,\qquad r\mapsto w^0(r),
\]
satisfying $ q^0(r):= u(0) - w^0(r) \in K_r$ for $\mu$-a.e. $r$,
we define, for $\mu$-a.e. $r$,
\[
w(r):= \Playop_{K_r}\left[u;q^0(r)\right] \in W^{1,1}(0,T;H).
\]
Whenever $r\mapsto c(r)w(r)$ is Bochner integrable in $W^{1,1}(0,T;H)$, we define the Prandtl--Ishlinskii operator
of play type associated with \(\mathfrak D\) by
\begin{equation}
 \PIop_{\mathfrak D}[u;w^0]
 :=a u + \int_{\Rcal} c(r)w(r)\,d\mu(r).
 \label{eq:continuous_pi_definition}
\end{equation}
When the constitutive data and initial memories are fixed, we use the abbreviations  $\PIop[u]:=\PIop_{\mathfrak D}[u;w^0]$ and \(\Playop_r[u]:=\Playop_{K_r}[u;q^0(r)]\).
\end{definition}

  The benefits of Definition~\ref{def:continuous_pi} are twofold. First, it includes the standard pointwise spatial realization considered in \cite[Section~IX.1]{visintinDifferentialModelsHysteresis1994}, obtained by taking  $\Rcal\subset [0,+\infty)$, $H=L^2(\Omega)$, and 
  \begin{equation}
    K_r=K_r^{\mathrm{pt}}:=\{z\in H:|z(x)|\leq r
    \text{ for a.e. }x\in\Omega\},
    \label{eq:standard_play_constraint_family}
  \end{equation}
  where  $\Omega\subset\mathbb R^d$ is a bounded Lipschitz domain.
  Second, it allows coupling between different spatial points through nonlocal play constraints, as illustrated in Example~\ref{ex:global_l2_ball}.

  \subsection{Practical assumptions and the parabolic hysteresis problem}
    The hysteresis relation introduced above serves as a material constitutive law, and the parabolic hysteresis problem arises by coupling \eqref{eq:continuous_pi_definition} with the diffusion balance. We impose the following assumptions, which are satisfied by standard PI models used in engineering applications. Representative examples are given in Appendix~\ref{app:examples}.
  \begin{assumption}
    \label{ass:continuous_model}
    The constitutive data $\mathfrak D$ in Definition~\ref{def:continuous_pi} satisfy the following conditions.  The measure space $(\Rcal,\mathcal A,\mu)$ is complete and finite with $\mu\geq0$, and the coefficient $c$ satisfies
    \begin{equation}
    S_c:=\int_{\Rcal} c(r)\,d\mu(r)<\infty.
    \label{eq:continuous_coefficient_moment}
    \end{equation}
    The graph of the set-valued map \(K\colon \Rcal \rightrightarrows H, \, r\mapsto K_r\) is measurable, i.e.,
    \begin{equation}
    \operatorname{Gr}K
    :=\{(r,z)\in\Rcal\times H:z\in K_r\}
    \in\mathcal A\otimes\mathcal B(H).
    \label{eq:continuous_graph_measurability}
    \end{equation}
    Here $\mathcal B(H)$ denotes the Borel $\sigma$-algebra of $H$.
  \end{assumption}

  The coefficient $c$ in Assumption~\ref{ass:continuous_model} induces the finite
  weighted measure
  \begin{equation}
  \nu_c(A):=\int_A c(r)\,d\mu(r),
  \qquad A\in\mathcal A,
  \label{eq:continuous_memory_measure}
  \end{equation}
  for which $\nu_c(\Rcal)=S_c$. Since $0<c(r)<\infty$, the measures $\nu_c$
  and $\mu$ have the same null sets. We use the weighted Bochner space
  \begin{equation}
  \Mmu:=L^2(\Rcal,\nu_c;H)
  =L^2\bigl(\Rcal,c(r)d\mu(r);H\bigr),
  \label{eq:continuous_memory_space}
  \end{equation}
  whose elements are strongly measurable maps $z:\Rcal\to H$, identified up
  to $\mu$-almost-everywhere equality, such that
  $\int_{\Rcal} c(r)\|z(r)\|^2\,d\mu(r)<\infty$. Its inner product and norm are
  \begin{equation}
  \langle z_1,z_2\rangle_{\Mmu}
  :=\int_{\Rcal} c(r)(z_1(r),z_2(r))_H\,d\mu(r),
  \qquad
  \|z\|_{\Mmu}:=\langle z,z\rangle_{\Mmu}^{1/2}.
  \label{eq:continuous_memory_product}
  \end{equation}
  With this inner product, \(\Mmu\) is a Hilbert space
  \cite[Section~1.2]{hytonen_analysis_2016}.

  From now on, let \(\Omega\subset\mathbb R^d\) be a bounded Lipschitz domain and set \(H:=L^2(\Omega)\) and \(V:=H_0^1(\Omega)\).
  We denote the norm on \(H\) by \(\|\cdot\|\) and equip \(V\) with the norm \(\|v\|_V:=\|\nabla v\|\).
  All time derivatives below are understood as Bochner weak derivatives in their respective Hilbert spaces.
  The weak formulation of \eqref{eq:pde_strong} is as follows:
  \begin{problem}[Parabolic hysteresis problem]
    \label{prob:parabolic_hysteresis_problem}
    Given $f\in L^2(0,T;H)$ and initial data $(u^0,w^0)\in V\times\Mmu$ satisfying $q^0(r):=u^0-w^0(r)\in K_r$ for $\mu$-a.e. $r\in\Rcal$, find $u\in H^1(0,T;H)\cap L^2(0,T;V)$ and $w\in H^1(0,T;\Mmu)$ such that the balance and play relations below hold a.e. in $(0,T)$, with the prescribed initial data:
    \begin{align}
      &\kappa(\partial_t u,\varphi)_H
    +\int_{\Rcal} c(r)(\partial_t w(r),\varphi)_H\,d\mu(r)
    +(\nabla u,\nabla\varphi)_H
    =(f,\varphi)_H,\quad
    \forall\varphi\in V,
    \label{eq:continuous_pi_balance}\\
    &q(r):=u-w(r)\in K_r,
    \qquad
    \partial_t w(r)\in N_{K_r}(q(r))
    \quad\text{for $\mu$-a.e. }r\in\Rcal,
    \label{eq:continuous_pi_play} \\
    & u(0) = u^0,\qquad
    w(0) = w^0,
    \label{eq:continuous_physical_initial_boundary_data}
    \end{align}
    where $\kappa:=1+a$.
  \end{problem}
\section{Subgradient-flow structure}
  \label{sec:continuous_pi_structure}

  Introducing \(\eta:=\partial_t w\), we rewrite the system \eqref{eq:continuous_pi_balance}--\eqref{eq:continuous_pi_play} as
  \begin{align}
    &\kappa(\partial_t u,\varphi)_H
    +\int_{\Rcal} c(r)(\eta(r),\varphi)_H\,d\mu(r)
    +(\nabla u,\nabla\varphi)_H
    =(f,\varphi)_H,\quad
    \forall\varphi\in V, \label{eq:continuous_pi_balance_eta}\\
    & (\partial_t w(r),\psi)_H -  (\eta(r),\psi)_H = 0,\ \forall\psi \in H,  \quad
    \eta(r)\in N_{K_r}(q(r))
    \;\text{for $\mu$-a.e. }r\in\Rcal. \label{eq:continuous_pi_play_eta}
  \end{align}
  We aim to express this coupled system as a convex subgradient flow
  for the state \(U=(u,w)\) in a suitable Hilbert space:
  \begin{equation}
    \partial_t U+\partial\Psi(U)\ni F
    \quad\text{in }\Hmu\quad\text{a.e. in }(0,T),
    \qquad U(0)=U^0.
    \label{eq:continuous_gradient_flow_formal}
  \end{equation}

  To this end, we identify the state space \(\Hmu\) and the convex energy functional \(\Psi\), characterize the subdifferential \(\partial\Psi\), and establish the equivalence between \eqref{eq:continuous_gradient_flow_formal} and the original hysteresis system \eqref{eq:continuous_pi_balance}--\eqref{eq:continuous_physical_initial_boundary_data}. The resulting formulation also yields the well-posedness and natural regularity properties needed in the subsequent error analysis.

  \subsection{State space, constraint and energy}
    Set \(\Hmu:=H\times\Mmu\) and endow this state space with
    \begin{equation}
    \langle (u,w),(v,z)\rangle_{\Hmu}
    :=\kappa(u,v)_H+\langle w,z\rangle_{\Mmu}.
    \label{eq:continuous_state_product}
    \end{equation}
    This inner product makes \(\Hmu\) a Hilbert space.
    For \(U=(u,w)\in\Hmu\), define its parameterized gap by
    \begin{equation}
    [\Gapop U](r):=u-w(r).
    \label{eq:continuous_gap_map}
    \end{equation}
    We also define the gap-constraint set
    \begin{equation}
    \mathcal K:=\{q\in\Mmu:q(r)\in K_r\text{ for $\mu$-a.e. }r\}.
    \label{eq:continuous_gap_constraint}
    \end{equation}

  \begin{proposition}[Boundedness and surjectivity of the gap operator]
  \label{prop:continuous_gap_map}
  The gap operator $\Gapop:\Hmu\to\Mmu$ is linear, bounded and surjective. More precisely,
  \begin{equation}
  \|\Gapop(u,w)\|_{\Mmu}^2
  \leq 2S_c\|u\|^2+2\|w\|_{\Mmu}^2
  \leq2\max\{S_c/\kappa,1\}\|(u,w)\|_{\Hmu}^2.
  \label{eq:continuous_gap_bound}
  \end{equation}
  \end{proposition}

  \begin{proof}
    For \((u,w)\in\Hmu\), the map \(r\mapsto u-w(r)\) is strongly
    measurable because \(w\) is strongly measurable. 
    Moreover,
    \[
      \int_{\Rcal}c(r)\|u-w(r)\|^2\,d\mu(r)
      \leq 2S_c\|u\|^2+2\|w\|_{\Mmu}^2.
    \]
    Thus \(\Gapop\) is well defined, and the definition of the norm
    on \(\Hmu\) yields \eqref{eq:continuous_gap_bound}.
    Linearity follows directly from the definition.
    Finally, \(\Gapop(0,-q)=q\) for every \(q\in\Mmu\), so \(\Gapop\) is surjective.
  \end{proof}

  \begin{proposition}[Closed convex gap constraint]
  \label{prop:continuous_gap_constraint}
  The set $\mathcal K$ is a nonempty closed convex subset of $\Mmu$.
  \end{proposition}
  \begin{proof}
  The zero map belongs to $\mathcal K$ because $0\in K_r$ for every $r$.
  Convexity follows pointwise from that of $K_r$. 
  To prove closedness, let $\{q_n\}_{n=1}^\infty\subset\mathcal K$ and suppose that $q_n\to q$ strongly in $\Mmu$.
  Passing to a subsequence if necessary, we have \(q_n(r)\to q(r)\) in \(H\) for \(\nu_c\)-a.e. \(r\).
  Since \(\nu_c\) and \(\mu\) have the same null sets and each \(K_r\) is closed, it follows that \(q(r)\in K_r\) for \(\mu\)-a.e. \(r\).
  Hence \(q\in\mathcal K\).
  \end{proof}

  Using the gap operator \(\Gapop\), we define the corresponding
  state constraint set
  \begin{equation}
  \mathcal C
  :=\Gapop^{-1}(\mathcal K)
  =\{(u,w)\in\Hmu:\Gapop(u,w)\in\mathcal K\}.
  \label{eq:continuous_state_constraint}
  \end{equation}
  By Propositions~\ref{prop:continuous_gap_map} and \ref{prop:continuous_gap_constraint}, \(\mathcal C\) is a nonempty closed convex subset of \(\Hmu\).
  Introduce the extended Dirichlet energy functional $E$ on $H$,
  \[
  E(u):=
  \begin{cases}
    \dfrac12\|\nabla u\|^2,&u\in V,\\
    +\infty,&u\notin V,
  \end{cases}
  \]
  and let $\pi_1:\Hmu\to H$ denote the first-component projection,
  $\pi_1(u,w)=u$. We set $\Phi:=E\circ\pi_1$ and define \(\Psi\) on $\Hmu$ by
  \begin{equation}
  \Psi:=\Phi+I_{\mathcal C},
  \label{eq:continuous_energy}
  \end{equation}
  where $I_{\mathcal C}$ is the indicator function of $\mathcal{C} \subset \Hmu$, i.e.,
  \[
    I_{\mathcal C}(u,w):=\begin{cases}0,&(u,w)\in {\mathcal C},\\+\infty,&(u,w)\notin {\mathcal C}.\end{cases}
  \]
  Thus the effective domain of the functional $\Psi:\Hmu\to(-\infty,+\infty]$ is
  \begin{equation}
  D(\Psi)
  =\{(u,w)\in\Hmu:u\in V,\ (u,w)\in\mathcal C\}.
  \label{eq:continuous_energy_domain}
  \end{equation}

  \begin{proposition}[Convex energy functional]
  \label{prop:continuous_convex_energy}
  The functional $\Psi$ is proper, convex and lower semicontinuous.
  \end{proposition}

  \begin{proof}
    The extended Dirichlet energy $E$ is proper, convex and lower semicontinuous on $H$ \cite{barbu_nonlinear_1976}.
    Since \(\pi_1\) is continuous and linear, \(\Phi=E\circ\pi_1\) is convex and lower semicontinuous on \(\Hmu\).
    The indicator \(I_{\mathcal C}\) is also convex and lower semicontinuous because \(\mathcal C\) is closed and convex.
    Therefore,
    \[
    \Psi=E\circ\pi_1+I_{\mathcal C}
    \]
    is convex and lower semicontinuous.
    Finally, \(\Psi(0,0)=0\), which proves that \(\Psi\) is proper.
  \end{proof}

  \begin{remark}
    The constraint functional $\Phi_3$ in \cite{little_semilinear_1994} is recovered as a special case of $I_{\mathcal C}$ when the family $K_r$ is chosen as the canonical pointwise constraint \eqref{eq:standard_play_constraint_family}, up to the corresponding sign and parameter conventions.
  \end{remark}

\subsection{Characterization of the subdifferential}
  We begin with the Dirichlet energy term $\Phi$, whose subdifferential follows from a straightforward extension to the product space.
  \begin{proposition}[Dirichlet energy term]
    \label{prop:continuous_Dirichlet_subdifferential}
    The domain of  $\partial\Phi$ is
    \begin{equation}
      D(\partial\Phi)
      =\{(u,w)\in\Hmu:u\in V,\ \Delta u\in H\}.
    \label{eq:continuous_Dirichlet_subdifferential_domain}
    \end{equation}
    Moreover, for every $U=(u,w)\in D(\partial\Phi)$,
    \begin{equation}
  \partial \Phi(U)
    =\left\{\left(-\dfrac1\kappa\Delta u,0\right)\right\}.
    \label{eq:continuous_Dirichlet_subdifferential}
    \end{equation}
  \end{proposition}

  \begin{proof}
    The projection $\pi_1:\Hmu\to H$ is bounded, linear, and surjective, so $D(E)-\operatorname{ran}\pi_1=H$.
    Hence the qualification condition of the linear-composition chain rule \cite[Corollary~16.53(i)]{bauschke_convex_2017} is satisfied, and
    \[
    \partial \Phi(U)
      =\partial(E\circ\pi_1)(U)=\pi_1^*\partial_HE(u),
    \]
    where \(\partial_HE\) denotes the subdifferential of \(E\) in \(H\).
    The weighted inner product \eqref{eq:continuous_state_product} gives $\pi_1^*h=(h/\kappa,0)$.  Moreover, the Dirichlet energy satisfies
    \[
      D(\partial_HE)=\{u\in V:\Delta u\in H\},
      \qquad \partial_HE(u)=\{-\Delta u\}.
    \]
    The stated formula and domain now follow.
  \end{proof}

  To characterize \(\partial I_{\mathcal C}\), we first introduce the aggregate operator \(\Aggop:\Mmu\to H\) by
  \begin{equation}
   \Aggop z:=\int_{\Rcal} c(r)z(r)\,d\mu(r),
   \label{eq:continuous_aggregate_operator}
  \end{equation}
  which is a bounded linear map, because
  \begin{equation}
   \|\Aggop z\|
   \leq\int_{\Rcal} c(r)\|z(r)\|\,d\mu(r)
   \leq S_c^{1/2}\|z\|_{\Mmu}.
   \label{eq:continuous_memory_aggregate_bound}
  \end{equation}
  With this notation, the adjoint of the gap operator is given by
  \begin{equation}
  \Gapop^*\eta =\left( \frac1\kappa\Aggop\eta, -\eta\right)\in\Hmu.
  \label{eq:continuous_gap_adjoint}
  \end{equation}
  Indeed, for every $(v,z)\in\Hmu$ and $\eta\in\Mmu$,
  \begin{align*}
  \langle\eta,\Gapop(v,z)\rangle_{\Mmu}
  &=\int_{\Rcal} c(r)(\eta(r),v-z(r))_H\,d\mu(r)\\
  &=\kappa\left(\frac{1}{\kappa}\Aggop\eta,v\right)_H+\langle-\eta,z\rangle_{\Mmu} = \left\langle \left(
    \frac1\kappa\Aggop\eta,
    -\eta\right), (v,z)\right\rangle_{\Hmu}.
  \end{align*}

  \begin{proposition}[State constraint term]
    \label{prop:continuous_state_normal}
    The domain of \(\partial I_{\mathcal C}\) is
    \[
      D(\partial I_{\mathcal C})=\mathcal C.
    \]
    Moreover, for \(U=(u,w)\in\mathcal C\), with
    \(q=\Gapop U\in\mathcal K\),
    \begin{equation}
      \partial I_{\mathcal C}(U)
      =\Gapop^*N_{\mathcal K}(q)
      =\left\{
      \left(\frac1\kappa\Aggop\eta,-\eta\right)\colon
        \eta\in N_{\mathcal K}(q)
      \right\}.
      \label{eq:continuous_state_normal}
    \end{equation}
  \end{proposition}

  \begin{proof}
    We have $I_{\mathcal C}=I_{\mathcal K}\circ\Gapop$. Since $\mathcal K$ is nonempty, closed and convex,
    $I_{\mathcal K}$ is proper, convex and lower semicontinuous on $\Mmu$.
    By Proposition~\ref{prop:continuous_gap_map}, \(\Gapop\) is bounded, linear and surjective.
    Hence \(D(I_{\mathcal K})-\operatorname{ran}\Gapop=\Mmu\), so the linear-composition chain rule \cite[Corollary~16.53(i)]{bauschke_convex_2017} gives
    \[
   \partial I_{\mathcal C}(U)
    =\Gapop^*\partial_{\Mmu}I_{\mathcal K}(\Gapop U)
    =\Gapop^*N_{\mathcal K}(q).
    \]
    The adjoint formula \eqref{eq:continuous_gap_adjoint} proves \eqref{eq:continuous_state_normal}.
    Since \(0\in N_{\mathcal K}(q)\) for every \(q\in\mathcal K\), the formula above also shows that \(D(\partial I_{\mathcal C})=\mathcal C\).
  \end{proof}

  To relate the subdifferential formulas above to the play relation \eqref{eq:continuous_pi_play_eta}, we next characterize \(N_{\mathcal K}(q)\) pointwise in the parameter \(r\).
  To this end, we introduce the metric projection $P_{K_r}\colon H\to K_r$ for $\mu$-a.e. $r\in\Rcal$, defined by
  \begin{equation}
  P_{K_r}z:=\operatorname*{argmin}_{v\in K_r}\|z-v\|,\quad z\in H.
  \label{eq:metric_projection_definition}
  \end{equation}
  The projection theorem \cite[Theorem~3.16]{bauschke_convex_2017} guarantees the existence and uniqueness of the minimizer. Moreover, \(p=P_{K_r}z\) if and only if
  \begin{equation}
  p\in K_r\ \text{ and }\ (z-p,v-p)_H\leq0
  \quad\forall v\in K_r.
  \label{eq:metric_projection_characterization}
  \end{equation}

  \begin{lemma}[Pointwise normal cone characterization]
    \label{lem:continuous_pointwise_normal}
    For every $q\in\mathcal K$,
    \begin{equation}
    N_{\mathcal K}(q)
    =\{\eta\in\Mmu\colon
          \eta(r)\in N_{K_r}(q(r))\,\text{ for $\mu$-a.e. }r\in\Rcal\}.
    \label{eq:continuous_pointwise_normal}
    \end{equation}
  \end{lemma}

  \begin{proof}
    If $\eta\in\Mmu$ and $\eta(r)\in N_{K_r}(q(r))$ for $\mu$-a.e. $r$, integrating the pointwise normal-cone inequality with respect to $\nu_c$ gives
    $\eta\in N_{\mathcal K}(q)$.
    Conversely, let $\eta\in N_{\mathcal K}(q)$ and define
    \[
      p(r):=P_{K_r}(q(r)+\eta(r))
      \quad\text{for \(\mu\)-a.e. \(r\)}.
    \]
    Since $q+\eta$ is strongly measurable, Proposition~\ref{prop:measurable_metric_projection} shows that $p$ is strongly
    measurable.
    Testing \eqref{eq:metric_projection_characterization}
    with \(v=q(r)\in K_r\) yields
    \[
      (\eta(r),p(r)-q(r))_H\geq\|p(r)-q(r)\|^2.
    \]
    By the Cauchy--Schwarz inequality,
    \[
      \|p(r)-q(r)\|\leq\|\eta(r)\|.
    \]
    Since \(q,\eta\in\Mmu\), it follows that \(p\in\Mmu\).
    Together with \(p(r)\in K_r\) for \(\mu\)-a.e. \(r\),
    this gives \(p\in\mathcal K\).
    Using \(p\) in the normal-cone inequality for
    \(\eta\in N_{\mathcal K}(q)\), we obtain
    \[
    \int_{\Rcal} c(r)(\eta(r),p(r)-q(r))_H\,d\mu(r)\leq0.
    \]
    Hence \(p(r)=q(r)\) for \(\mu\)-a.e. \(r\).
    By \eqref{eq:metric_projection_characterization},
    the identity \(q(r)=P_{K_r}(q(r)+\eta(r))\) is equivalent to
    \(\eta(r)\in N_{K_r}(q(r))\), proving the reverse inclusion.
  \end{proof}

  \begin{proposition}[Subdifferential of \(\Psi\)]
    \label{prop:continuous_subdifferential}
    The domain of $\partial\Psi$ is
    \begin{equation}
      D(\partial\Psi)
      =\left\{(u,w)\in\Hmu:u\in V,\ (u,w)\in\mathcal C,\ \Delta u\in H\right\}.
      \label{eq:continuous_operator_domain}
      \end{equation}
    For $U=(u,w)\in D(\partial\Psi)$, with $q=\Gapop U \in \mathcal K$, the subdifferential of $\Psi$ is
    \begin{equation}
    \partial\Psi(U)
    =\left\{
    \left(
      \frac1\kappa\left(-\Delta u+\Aggop\eta\right),
      -\eta
    \right):
    \begin{array}{l}
      \eta\in\Mmu,\\
      \eta(r)\in N_{K_r}(q(r))\;\text{for $\mu$-a.e. }r\in\Rcal
    \end{array}
    \right\}.
    \label{eq:continuous_subdifferential}
    \end{equation}
  \end{proposition}

  \begin{proof}
    Both $\Phi$ and $I_{\mathcal C}$ are proper, convex and lower
    semicontinuous on $\Hmu$. We verify the qualification condition for
    the sum rule \cite[Corollary~16.48(i)]{bauschke_convex_2017} by showing
    \begin{equation}
      D(\Phi)-D(I_{\mathcal C})=\Hmu.
      \label{eq:continuous_sum_qualification}
    \end{equation}
    For any \((v,z)\in\Hmu\), define the constant memory map \(\mathbf v(r):=v\). Since
    \(
      \|\mathbf v\|_{\Mmu}^2=S_c\|v\|^2<\infty,
    \)
    we have \(\mathbf v\in\Mmu\).
    Then \((0,z-\mathbf v)\in D(\Phi)=V\times\Mmu\), while \((-v,-\mathbf v)\in\mathcal C\) because its gap is zero. Consequently,
    \[
      (v,z)=(0,z-\mathbf v)-(-v,-\mathbf v)
      \in D(\Phi)-\mathcal C,
    \]
    which proves \eqref{eq:continuous_sum_qualification}.
    Thus \(0\) belongs to the strong relative interior of
    \(D(\Phi)-D(I_{\mathcal C})\), and the sum rule gives
    \begin{equation}
      \partial\Psi(U)=\partial\Phi(U)+\partial I_{\mathcal C}(U).
      \label{eq:continuous_sum_rule}
    \end{equation}
    Propositions~\ref{prop:continuous_Dirichlet_subdifferential} and~\ref{prop:continuous_state_normal}, together with Lemma~\ref{lem:continuous_pointwise_normal}, yield \eqref{eq:continuous_operator_domain} and \eqref{eq:continuous_subdifferential}.
  \end{proof}

\subsection{Equivalence and well-posedness}

\begin{theorem}[Equivalence]
\label{thm:continuous_physical_equivalence}
  Let \(f\) and \(U^0=(u^0,w^0)\) satisfy the assumptions of
  Problem~\ref{prob:parabolic_hysteresis_problem}.
  Let \(U=(u,w)\) satisfy $u\in H^1(0,T;H)\cap L^2(0,T;V)$ and $w\in H^1(0,T;\Mmu)$. Then
  \eqref{eq:continuous_pi_balance}--\eqref{eq:continuous_physical_initial_boundary_data}
  hold if and only if
  \begin{equation}
  \partial_t U+\partial\Psi(U)\ni F
  \quad\text{in }\Hmu\quad\text{a.e. in }(0,T),
  \qquad U(0)=U^0,
  \label{eq:continuous_gradient_flow}
  \end{equation}
  where $F:=(f/\kappa,0)\in L^2(0,T;\Hmu)$.
\end{theorem}

\begin{proof}
Suppose first that \eqref{eq:continuous_gradient_flow} holds. By Proposition~\ref{prop:continuous_subdifferential}, for a.e. \(t\in(0,T)\), we have \(U(t)\in D(\partial\Psi)\) and there exists \(\eta(t)\in\Mmu\) such that
\(q(t,r):=u(t)-w(t,r)\in K_r\), \(\eta(t,r)\in N_{K_r}(q(t,r))\)
for \(\mu\)-a.e. \(r\in\Rcal\), and
\[
\partial_tu(t)+\frac1\kappa\bigl(-\Delta u(t)+\Aggop\eta(t)\bigr)=\frac{f(t)}{\kappa} \quad\text{in }H,
\qquad
\partial_tw(t)-\eta(t)=0 \quad\text{in }\Mmu.
\]
Thus, using $\eta=\partial_t w$, we multiply the first identity by \(\kappa\) and test it with an arbitrary $\varphi\in V$ to obtain \eqref{eq:continuous_pi_balance}, while the pointwise normal-cone inclusion gives \eqref{eq:continuous_pi_play}.

Conversely, suppose that \eqref{eq:continuous_pi_balance}--\eqref{eq:continuous_physical_initial_boundary_data} hold, and set \(\eta:=\partial_t w\in L^2(0,T;\Mmu)\).
The play relation \eqref{eq:continuous_pi_play} gives \(U(t)\in\mathcal C\) and
\(\eta(t,r)\in N_{K_r}(q(t,r))\) for \(\mu\)-a.e. \(r\),
for a.e. \(t\in(0,T)\).
The weak balance \eqref{eq:continuous_pi_balance} yields,
in the sense of distributions,
\begin{equation}\label{eq:continuous_derived_laplacian}
  \Delta u=\kappa \partial_t u+\Aggop \partial_t w-f \in L^2(0,T;H).
\end{equation}
Hence \(U(t)\in D(\partial\Psi)\) for a.e. \(t\in(0,T)\).
Since all terms belong to $H$ and $V$ is dense in $H$, the weak balance
\eqref{eq:continuous_pi_balance} is equivalent to the first component of
\eqref{eq:continuous_gradient_flow}; the identity $\partial_t w-\eta=0$ is its
second component.

\end{proof}

\begin{remark}
  The state augmentation is possible because the current play state
  retains all information from the past input relevant to the subsequent evolution.
  Thus, although the play operator is history dependent as an input--output mapping, its representation in the augmented state space is local in time.
  The PI superposition inherits this property through its parameterized family of play states.
\end{remark}

\begin{theorem}[Well-posedness and natural regularity]
\label{thm:continuous_wellposedness}
  Suppose Assumption~\ref{ass:continuous_model} holds. Problem~\ref{prob:parabolic_hysteresis_problem}
  has a unique solution with
  \begin{align}
    &u\in H^1(0,T;H)\cap L^\infty(0,T;V), \quad \Delta u\in L^2(0,T;H),
    \quad
    w\in H^1(0,T;\Mmu).
    \label{eq:continuous_natural_time_regularity}
  \end{align}
  Moreover, for every \(t\in[0,T]\),
  \begin{equation}
    \int_0^t \left(\kappa\|\partial_t u(s)\|^2+\|\partial_t w(s)\|_{\Mmu}^2\right)\,ds+
    \|\nabla u(t)\|^2
    \le
    \|\nabla u^0\|^2
    +\frac1{\kappa}\|f\|_{L^2(0,t;H)}^2. \label{eq:continuous_aprior}
  \end{equation}
\end{theorem}

\begin{proof}
Proposition~\ref{prop:continuous_convex_energy} shows that $\Psi$ is proper, convex and lower semicontinuous on $\Hmu$, so $\partial\Psi$ is maximal monotone.
Since \(U^0\in D(\Psi)\subset\overline{D(\partial\Psi)}\), Theorem~4.2 of \cite{barbu_nonlinear_2010} yields a unique mild solution \(U\in C([0,T];\Hmu)\) of \eqref{eq:continuous_gradient_flow} with \(U(0)=U^0\).
Moreover, $U^0\in D(\Psi)$ and $F\in L^2(0,T;\Hmu)$, so Theorem~4.11 of \cite{barbu_nonlinear_2010} yields \eqref{eq:continuous_gradient_flow} with
\[
  U(t) \in D(\partial \Psi) \text{ for a.e. }t\in(0,T), \quad U\in H^1(0,T;\Hmu), \quad \Psi(U)\in W^{1,1}(0,T),\]
together with the energy estimate
\[\frac12\int_0^t\|\partial_t U(s)\|_{\Hmu}^2\,ds
+\Psi(U(t))
\le
\Psi(U^0)
+\frac12\int_0^t\|F(s)\|_{\Hmu}^2\,ds, \quad t\in[0,T].\]
In particular, $\Psi(U(t)) = \frac12 \|\nabla u(t) \|^2$ and hence $u\in L^\infty(0,T;V)$.
Thus all assumptions of Theorem~\ref{thm:continuous_physical_equivalence} are satisfied and the unique solution of the gradient-flow problem \eqref{eq:continuous_gradient_flow} coincides with the unique solution of Problem~\ref{prob:parabolic_hysteresis_problem}.
\end{proof}

\begin{remark}
For pointwise hysteresis laws, Visintin
\cite[Chapter~VIII]{visintinDifferentialModelsHysteresis1994}
develops $m$-accretive formulations and integral-solution results.
Theorem~IX.2.6 establishes strong time regularity for Lipschitz
generalized play, and extensions to continuous generalized
Prandtl--Ishlinskii operators are discussed in the remarks following
Theorem~IX.2.7.
\end{remark}

\section{Temporal discretization and error estimates}
\label{sec:continuous_pi_time}

  We now apply the implicit Euler method to the subgradient-flow formulation established in Section~\ref{sec:continuous_pi_structure}. We first derive a posteriori error estimates for arbitrary time partitions and then use them to obtain a priori convergence rates.

  \subsection{Implicit Euler discretization}

Let
\[
 0=t_0<t_1<\cdots<t_N=T,\qquad
 \tau_n=t_n-t_{n-1},\qquad
 \tau=\max_{1\leq n\leq N}\tau_n,
\]
be an arbitrary partition and set $\delta_nz^n=(z^n-z^{n-1})/\tau_n$ for any sequence $(z^n)_{n=0}^N$.
For $n=0,\ldots,N$, write
\[
 U_\tau^n=(u_\tau^n,w_\tau^n),\qquad F^n=(f^n/\kappa,0),
\]
where $U_\tau^0\in D(\Psi)$ is a prescribed initial state and $f^n\in H$ is the chosen approximation of the load. The implicit Euler discretization of \eqref{eq:continuous_gradient_flow} is to seek a sequence $U_\tau^n\in D(\partial\Psi)$, $n=1,\ldots,N$, satisfying
\begin{equation}
 \delta_nU_\tau^n+\partial\Psi(U_\tau^n)\ni F^n
 \quad\text{in }\Hmu.
 \label{eq:continuous_abstract_euler}
\end{equation}
Since $\partial\Psi$ is maximal monotone, its resolvent is everywhere defined, single-valued, and nonexpansive \cite{barbu_nonlinear_2010}. Hence,
\[
 U_\tau^n=(I+\tau_n\partial\Psi)^{-1}
 (U_\tau^{n-1}+\tau_nF^n)
\]
defines the unique solution of \eqref{eq:continuous_abstract_euler} and places it in $D(\partial\Psi)$.

For each $n=1,\ldots,N$, \eqref{eq:continuous_abstract_euler}
is equivalent to finding $(u_\tau^n,w_\tau^n)\in V\times\Mmu$
such that
\begin{gather}
  \kappa(\delta_nu_\tau^n,\varphi)_H+
 (\Aggop\delta_nw_\tau^n,\varphi)_H
 +(\nabla u_\tau^n,\nabla\varphi)_H
 =(f^n,\varphi)_H,\quad
 \forall\varphi\in V,\\
 q_\tau^n(r):=u_\tau^n-w_\tau^n(r) \in K_r ,\quad \delta_nw_\tau^n(r)\in
 N_{K_r}\left(q_\tau^n(r)\right)
 \quad\text{for $\mu$-a.e. }r\in\Rcal.
 \label{eq:continuous_physical_euler_play}
\end{gather}
The normal-cone inclusion in \eqref{eq:continuous_physical_euler_play} is equivalent to
\[
 q_\tau^n(r)=P_{K_r}\left(q_\tau^n(r)+\tau_n\delta_nw_\tau^n(r)\right).
\]
Hence the hysteresis unknown has the explicit formula
\begin{equation}
 w_\tau^n(r)=u_\tau^n-P_{K_r}\left(u_\tau^n-w_\tau^{n-1}(r)\right)
 \quad\text{for $\mu$-a.e. }r\in\Rcal,
\end{equation}
which is more convenient for numerical implementation.
\begin{scheme}[Implicit Euler scheme]
  Given $(u_\tau^0,w_\tau^0)\in D(\Psi)$ and $f^n\in H$, seek $u_\tau^n\in V$, $n=1,\ldots,N$, satisfying
  \begin{gather}
    \kappa(\delta_nu_\tau^n,\varphi)_H+
   (\Aggop\delta_nw_\tau^n,\varphi)_H
   +(\nabla u_\tau^n,\nabla\varphi)_H
   =(f^n,\varphi)_H,\quad
   \forall\varphi\in V,
   \label{eq:continuous_physical_euler_balance}\\
   w_\tau^n(r)=u_\tau^n-P_{K_r}\left(u_\tau^n-w_\tau^{n-1}(r)\right), \quad\text{for $\mu$-a.e. }r\in\Rcal,
   \label{eq:continuous_projection_update_w}
  \end{gather}
\end{scheme}

\begin{remark}
  For the characteristic \eqref{eq:standard_play_constraint_family}, the metric projection is the pointwise clipping operator: for every $z\in H$,
  \[
    [P_{K_r}z](x)=\min\{r,\max\{-r,z(x)\}\}
    \quad\text{for a.e. }x\in\Omega.
  \]
  Hence \eqref{eq:continuous_projection_update_w} becomes
  \begin{equation*}
    w_\tau^n(r)
    = \max\{u_\tau^n-r,\min\{u_\tau^n+r,w_\tau^{n-1}(r)\}\},
  \end{equation*}
  which is exactly the update formula for play operators in \cite[Example~2.1.2]{brokateHysteresisPhaseTransitions1996}. For general characteristics, evaluating the projection may require
  solving the minimization problem \eqref{eq:metric_projection_definition}.
\end{remark}

On $I_n=(t_{n-1},t_n]$, define the affine and piecewise-constant reconstructions by
\begin{align}
 &U_\tau(t)=\frac{t_n-t}{\tau_n}U_\tau^{n-1}
   +\frac{t-t_{n-1}}{\tau_n}U_\tau^n,
 \quad \overline U_\tau(t)= (\overline u_\tau(t),  \overline w_\tau(t))=U_\tau^n,
 \label{eq:continuous_state_reconstructions}\\
 & \overline F(t)=F^n,
 \qquad \overline f(t)=f^n.
 \label{eq:continuous_force_reconstruction}
\end{align}
At $t=0$, set $U_\tau(0)=\overline U_\tau(0)=U_\tau^0$, $\overline F(0)=F^0$ and $\overline f(0)=f^0$.
By convexity of $D(\Psi)$, the affine reconstruction satisfies
$U_\tau(t)\in D(\Psi)$ for every $t\in[0,T]$, whereas
$\overline U_\tau(t)\in D(\partial\Psi)$ for every $t\in(0,T]$.

To control the error in the energy norm, we introduce the quantity
$\sigma$ from \cite[Definition~2.3]{nochetto_posteriori_2000}.
For $U=(u,w)\in D(\partial\Psi)$ and $\widehat U=(\widehat u,\widehat w)\in D(\Psi)$, set
\begin{equation}
  \sigma(U;\widehat U)
  :=
  \Psi(\widehat U)-\Psi(U)
  -\sup_{\Xi\in\partial\Psi(U)}
  \langle\Xi,\widehat U-U\rangle_{\Hmu}.
  \label{eq:continuous_sigma}
\end{equation}
For every \(\Xi\in\partial\Psi(U)\), the definition
\eqref{eq:continuous_sigma} immediately gives
\begin{equation}
 \Psi(\widehat U)-\Psi(U)
 -\langle\Xi,\widehat U-U\rangle_{\Hmu}
 \geq \sigma(U;\widehat U).
 \label{eq:continuous_sigma_bound}
\end{equation}
For the present energy, Proposition~\ref{prop:continuous_subdifferential}
gives
\[
 \Psi(\widehat U)-\Psi(U)
 -\langle\Xi,\widehat U-U\rangle_{\Hmu}
 =
 \frac12\|\nabla(\widehat u-u)\|^2
 -\langle\eta,\widehat q-q\rangle_{\Mmu},
\]
where $q=\Gapop U$, $\widehat q=\Gapop\widehat U$ and
$\eta\in N_{\mathcal K}(q)$ corresponds to $\Xi$.
Since \(\widehat q\in\mathcal K\), the normal-cone inequality gives
\(\langle\eta,\widehat q-q\rangle_{\Mmu}\leq0\),
with equality when $\eta=0$.
Taking the infimum over $\Xi\in\partial\Psi(U)$ therefore yields
\begin{equation}
 \sigma(U;\widehat U)
 =
 \frac12\|\nabla(\widehat u-u)\|^2.
 \label{eq:continuous_sigma_identity}
\end{equation}

Let $\mathfrak{E}:=\max\{\mathfrak E_H,\mathfrak E_\sigma\}$, where
\begin{align*}
  \mathfrak E_H
  &:=\max_{t\in[0,T]}
  \|U(t)-U_\tau(t)\|_{\Hmu},\\
  \mathfrak E_\sigma
  &:=\left(
  2\int_0^T
  \bigl[
  \sigma(U(t);U_\tau(t))
  +\sigma(\overline U_\tau(t);U(t))
  \bigr]\,dt
  \right)^{1/2}\\
  &=\left(
  \|u-u_\tau\|_{L^2(0,T;V)}^2
  +\|u-\overline u_\tau\|_{L^2(0,T;V)}^2
  \right)^{1/2}.
\end{align*}
Since $\kappa\geq1$,
\begin{align*}
 &\mathfrak E^*:=\max\Bigg\{
 \max_{t\in[0,T]}\left(\|u(t)-u_\tau(t)\|^2
 +\|w(t)-w_\tau(t)\|_{\Mmu}^2\right)^{1/2},\notag\\
 &\hspace{3em}\left(\|u-u_\tau\|_{L^2(0,T;V)}^2
 +\|u-\overline u_\tau\|_{L^2(0,T;V)}^2\right)^{1/2}\Bigg\}
 \leq \mathfrak E.
\end{align*}
Consequently, the abstract estimates of \cite{nochetto_error_1998,nochetto_posteriori_2000} for \(\mathfrak E\) directly control the desired time-discretization error \(\mathfrak E^*\).

\subsection{A posteriori error estimates}
We follow the a posteriori analysis in \cite{nochetto_posteriori_2000}.
Although its abstract setting assumes a separable Hilbert space, \eqref{eq:continuous_sigma_identity} directly ensures the measurability of the terms involving $\sigma$.
Thus the separability argument in its Remark~2.4 is unnecessary here.

For $1\leq n\leq N$, define
\begin{equation}
  \begin{aligned}
    E_n:=&
    \langle F^n-\delta_nU_\tau^n,\delta_nU_\tau^n\rangle_{\Hmu}
    -\frac{\Psi(U_\tau^n)-\Psi(U_\tau^{n-1})}{\tau_n} \\
    =&(f^n-\kappa\delta_nu_\tau^n,\delta_nu_\tau^n)_H
  -\|\delta_nw_\tau^n\|_{\Mmu}^2
  -\frac{\|\nabla u_\tau^n\|^2-\|\nabla u_\tau^{n-1}\|^2}{2\tau_n},
  \end{aligned}
 \label{eq:continuous_estimator_En}
\end{equation}
which measures the defect in the discrete chain rule.
Since $F^n-\delta_nU_\tau^n\in\partial\Psi(U_\tau^n)$, 
\eqref{eq:continuous_sigma_bound} and
\eqref{eq:continuous_sigma_identity} imply
\begin{equation}
 \tau_nE_n\geq
 \frac12\|\nabla(u_\tau^n-u_\tau^{n-1})\|^2\geq0.
 \label{eq:continuous_En_nonnegative}
\end{equation}

\begin{theorem}[A posteriori error estimate]
\label{thm:continuous_aposteriori}
Under Assumption~\ref{ass:continuous_model}, let $U=(u,w)$ be the exact solution of \eqref{eq:continuous_pi_balance}-- \eqref{eq:continuous_physical_initial_boundary_data}, and let $U_\tau^n=(u_\tau^n,w_\tau^n)$, $n=0,\ldots,N$, solve the implicit Euler scheme \eqref{eq:continuous_physical_euler_balance}-- \eqref{eq:continuous_projection_update_w}. Then
\begin{equation*}
  \mathfrak E^*\leq
  \left(
  \kappa\|u(0)-u_\tau^0\|^2
  +\|w(0)-w_\tau^0\|_{\Mmu}^2
  +\sum_{n=1}^N\tau_n^2E_n
  \right)^{1/2}
  +\frac1{\sqrt\kappa}
  \|f-\overline f\|_{L^1(0,T;H)}.
 \end{equation*}
\end{theorem}

\begin{proof}
  By Theorem~\ref{thm:continuous_wellposedness}, \(U\) is a strong solution of
  \eqref{eq:continuous_gradient_flow}. The energy \(\Psi\) is proper, convex,
  and lower semicontinuous by Proposition~\ref{prop:continuous_convex_energy},
  and \eqref{eq:continuous_abstract_euler} is the implicit Euler inclusion
  in \(\Hmu\). Moreover, \(E_n\) in \eqref{eq:continuous_estimator_En}
  agrees with the estimator in \cite[Definition~3.1]{nochetto_posteriori_2000}.
  Following the proof of \cite[Theorem~3.2]{nochetto_posteriori_2000}
  therefore gives
  \[
  \mathfrak E\leq
  \left(\|U(0)-U_\tau^0\|_{\Hmu}^2
  +\sum_{n=1}^N\tau_n^2E_n\right)^{1/2}
  +\|F-\overline F\|_{L^1(0,T;\Hmu)}.
  \]
  The result follows from \(\mathfrak E^*\leq\mathfrak E\) and
  \(\|F-\overline F\|_{L^1(0,T;\Hmu)}
  =\kappa^{-1/2}\|f-\overline f\|_{L^1(0,T;H)}\).
\end{proof}

Following \cite[Definition~3.10]{nochetto_posteriori_2000},
we introduce a second estimator based on the monotonicity
of $\partial\Psi$.
For $2\leq n\leq N$, set
\begin{equation}\label{eq:continuous_estimator_second_difference}
  \delta_n^2U_\tau^n
  =\frac{\delta_nU_\tau^n-\delta_{n-1}U_\tau^{n-1}}{\tau_n},
\end{equation}
and define the residual monotone terms
\begin{equation}
  \begin{aligned}
    D_n:=&\tau_n \langle\delta_nF^n-\delta_n^2U_\tau^n, \delta_nU_\tau^n\rangle_{\Hmu} \\
    =&\tau_n\left(
  (\delta_nf^n-\kappa\delta_n^2u_\tau^n,\delta_nu_\tau^n)_H
  -\langle\delta_n^2w_\tau^n,\delta_nw_\tau^n\rangle_{\Mmu}\right).
  \end{aligned}
 \label{eq:continuous_estimator_Dn}
\end{equation}
The inclusion \eqref{eq:continuous_abstract_euler} at time step $n-1$ gives
\begin{equation}
  \Psi(U_\tau^n)-\Psi(U_\tau^{n-1}) \geq \left\langle F^{n-1}-\delta_{n-1}U_\tau^{n-1} , U_\tau^n - U_\tau^{n-1} \right\rangle_{\Hmu}. \label{eq:time_semidiscrete_subgradient_inequality}
\end{equation}
Hence, for $2\leq n\leq N$,
\begin{equation}
  0\leq E_n\leq D_n.
  \label{eq:continuous_En_le_Dn}
\end{equation}

\begin{corollary}[A posteriori estimate with residual monotone terms]
\label{cor:continuous_aposteriori_Dn}
  Under the assumptions of Theorem~\ref{thm:continuous_aposteriori}, assume in addition that $\Delta u_\tau^0\in H$. Define an auxiliary initial velocity by
  \begin{equation}
    \delta_0U_\tau^0 := \left(\frac{f^0+\Delta u_\tau^0}{\kappa},0\right) .
    \label{eq:initial_difference_quotient}
  \end{equation}
  With this choice, extend
  \eqref{eq:continuous_estimator_second_difference} and
  \eqref{eq:continuous_estimator_Dn} to $n=1$.
  Then $D_1$ is well defined and computable, and
  \begin{equation*}
  \mathfrak E^*\leq
  \left(\kappa\|u(0)-u_\tau^0\|^2
  +\|w(0)-w_\tau^0\|_{\Mmu}^2 + \sum_{n=1}^N\tau_n^2D_n\right)^{1/2}
  +\frac1{\sqrt\kappa}
  \|f-\overline f\|_{L^1(0,T;H)}.
  \end{equation*}
\end{corollary}

\begin{proof}
 The assumptions guarantee $U_\tau^0 \in D(\partial \Psi)$ and $\delta_0U_\tau^0 \in  F^0-\partial\Psi(U_\tau^0)$. Hence
 \eqref{eq:time_semidiscrete_subgradient_inequality} holds for $n=1$ and \eqref{eq:continuous_En_le_Dn} holds for all $1\leq n\leq N$.
 The estimate now follows from Theorem~\ref{thm:continuous_aposteriori}.
\end{proof}

\subsection{A priori convergence rates}

For a real Hilbert space $X$, let $BV_r(0,T;X)$ denote the space of $X$-valued functions $g\colon[0,T]\to X$ of bounded variation that are right-continuous on $[0,T)$ and satisfy $g(T)=g(T-)$.
This extends the convention of \cite[Section~XII.7]{visintinDifferentialModelsHysteresis1994} to $X$-valued functions.
We write $\operatorname{Var}(g;[0,T])$ for the total variation of $g$.

\begin{corollary}[Half-order temporal error estimate]
\label{cor:continuous_half_order}
Suppose Assumption~\ref{ass:continuous_model} holds. Let $U=(u,w)$ be the exact solution to Problem~\ref{prob:parabolic_hysteresis_problem} and let $U_\tau^n=(u_\tau^n,w_\tau^n)$, $n=0,\ldots,N$, be the implicit Euler solution initialized by $U_\tau^0=U^0 \in D(\Psi)$. Set $f^0=0$ and $f^n = \frac{1}{\tau_n}\int_{t_{n-1}}^{t_n} f(s)\, \mathrm{d} s$, $n=1,\ldots,N$. Then
\begin{align}
 &\mathfrak E^*\leq\left[
 3\tau\left(
 \frac12\|\nabla u^0\|^2+
 \frac1\kappa\|f\|_{L^2(0,T;H)}^2\right)
 \right]^{1/2}.
 \label{eq:continuous_half_order}
\end{align}
\end{corollary}

\begin{proof}
Apply \cite[Theorem~3.16]{nochetto_posteriori_2000}, using \(\Psi(U^0)=\frac12\|\nabla u^0\|^2\) and \(\|F\|_{L^2(0,T;\Hmu)}^2=\frac1\kappa\|f\|_{L^2(0,T;H)}^2\).
\end{proof}

\begin{theorem}[First-order temporal error estimate]
\label{thm:continuous_first_order}
Under Assumption~\ref{ass:continuous_model}, let $U=(u,w)$ be the exact solution to Problem~\ref{prob:parabolic_hysteresis_problem}  and let $U_\tau^n=(u_\tau^n,w_\tau^n)$, $n=0,\ldots,N$, be the implicit Euler solution initialized by $U_\tau^0=U^0$. Assume additionally that
\[
 \Delta u^0\in H,\qquad f\in BV_r(0,T;H).
\]
Sample $f^n=f(t_n)$ for $n=0,\ldots,N$, using the left trace at $T$. Then
\begin{align}
 &\mathfrak E^*\leq\frac{\tau}{\sqrt\kappa}
 \left[
 \frac1{\sqrt2}\|f(0)+\Delta u^0\|+
 2\operatorname{Var}(f;[0,T])
 \right].
 \label{eq:continuous_first_order}
\end{align}
\end{theorem}

\begin{proof}
  Since $U_\tau^0=U^0\in D(\partial\Psi)$ by
  \eqref{eq:continuous_operator_domain}, the choice
  \eqref{eq:initial_difference_quotient} satisfies
  $\delta_0U_\tau^0\in F^0-\partial\Psi(U_\tau^0)$.
  The proof of \cite[Theorem~3.18]{nochetto_posteriori_2000}
  uses only this inclusion for the initial velocity and therefore
  applies to our choice. With the BV representative of $f$ sampled
  as specified above, it yields
  \begin{equation}
    \left(\sum_{n=1}^ND_n\right)^{1/2}
    \leq\frac1{\sqrt2}\|\delta_0U_\tau^0\|_{\Hmu}
    +\operatorname{Var}(\overline F;[0,T]).
    \label{eq:continuous_Dn_BV_control}
  \end{equation}
  The estimate follows from Corollary~\ref{cor:continuous_aposteriori_Dn}, where \eqref{eq:initial_difference_quotient} gives
  \[
    \|\delta_0U_\tau^0\|_{\Hmu}
    \leq\frac1{\sqrt\kappa}\|f^0+\Delta u^0\|,
  \]
  together with
  \(
    \operatorname{Var}(\overline F;[0,T])\leq \kappa^{-1/2}\operatorname{Var}(f;[0,T])
  \),
  \(\|f-\overline f\|_{L^1(0,T;H)} \leq \tau\operatorname{Var}(f;[0,T]) \).
\end{proof}

\begin{remark}[Constitutive-output error]
\label{rem:continuous_weighted_control}
Let $y=\PIop[u]$ and define the reconstructed discrete output by
$y_\tau:=a u_\tau+\Aggop w_\tau$. Then
\[
 \|y-y_\tau\|
 \leq a\|u-u_\tau\|+S_c^{1/2}\|w-w_\tau\|_{\Mmu}.
\]
Hence the preceding estimates also control the PI output in $L^\infty(0,T;H)$.
\end{remark}

\section{Spatial and space--time error estimates}
\label{sec:continuous_pi_space}

  We now turn to the spatial approximation and its combination with the temporal discretization. The key point is that the diffusive field is approximated in a finite element space, while the hysteresis variables remain in their natural memory space. The spatial analysis is driven by the approximation properties of the diffusive field, whose \(H^2\)-regularity follows from the natural regularity established above and standard elliptic regularity.

Throughout this section, let \(\Omega\) be a bounded convex polytopal domain. The Dirichlet regularity result
\cite[Theorem~3.2.1.2]{grisvard_elliptic_1985}
gives a constant $C_{\rm ell}>0$, depending only on $\Omega$,
such that
\begin{equation}
 \|v\|_{H^2(\Omega)}
 \leq C_{\rm ell}\|\Delta v\|
 \quad\text{for }v\in V,\ \Delta v\in H.
 \label{eq:continuous_elliptic_regularity}
\end{equation}
Theorem~\ref{thm:continuous_wellposedness} already gives $\Delta u\in L^2(0,T;H)$. Therefore \eqref{eq:continuous_elliptic_regularity} immediately yields
\begin{equation}
 \|u\|_{L^2(0,T;H^2(\Omega))}
 \leq C_{\rm ell}\|\Delta u\|_{L^2(0,T;L^2(\Omega))}.
 \label{eq:continuous_natural_H2}
\end{equation}

Let $\{V_h\}$ be a family of finite-dimensional conforming subspaces of $V$, and let $R_h:V\to V_h$ be the Ritz projection:
\begin{equation}
 (\nabla R_hv,\nabla z_h)_H=(\nabla v,\nabla z_h)_H
 \qquad\forall z_h\in V_h.
 \label{eq:continuous_Ritz_projection}
\end{equation}
Assume that there is a constant $C_{\rm app}>0$, independent of $h$ and $v$, such that
\begin{align}
 &\|v-R_hv\|+h\|\nabla(v-R_hv)\|
 \leq C_{\rm app}h^2\|v\|_{H^2(\Omega)},
 \label{eq:continuous_Ritz_H2}\\
 &\|v-R_hv\|
 \leq C_{\rm app}h\|v\|_V.
 \label{eq:continuous_Ritz_energy}
\end{align}
The first estimate holds for every $v\in V\cap H^2(\Omega)$, and the second for every $v\in V$. $C_{\rm app}$ may depend on $\Omega$ and on the chosen family $\{V_h\}$.

A standard example is the conforming $P_1$ space on a shape-regular
family of conforming simplicial meshes $\{\mathcal T_h\}$ of $\Omega$:
\[
h:=\max_{T\in\mathcal T_h}\operatorname{diam}(T),\qquad
V_h:=\{v_h\in H_0^1(\Omega):
       v_h|_T\in\mathbb P_1(T)
       \text{ for every }T\in\mathcal T_h\}.
\]
For this choice, \eqref{eq:continuous_Ritz_H2} and
\eqref{eq:continuous_Ritz_energy} hold.

\subsection{Restricted subgradient flow}

Define the restricted state space
\begin{equation}
  \mathcal H_h:=V_h\times\Mmu
  \subset \Hmu,
  \label{eq:continuous_discrete_state_space}
\end{equation}
endowed with the inner product inherited from $\Hmu$.
Since $V_h$ is a finite-dimensional, hence closed, subspace of $H$,
$\mathcal H_h$ is a Hilbert space.
Let
\[
  \Gapop_h:=\Gapop|_{\mathcal H_h}
  :\mathcal H_h\to\Mmu,
  \qquad
  \Gapop_h(u_h,w)=u_h-w.
\]
The operator $\Gapop_h$ is also linear, bounded and surjective.
We then define the restricted state constraint
\begin{equation}
  \mathcal C_h
  :=\Gapop_h^{-1}(\mathcal K)
  =\{(u_h,w)\in\mathcal H_h:
      \Gapop_h(u_h,w)\in\mathcal K\}.
  \label{eq:continuous_discrete_constraint}
\end{equation}
Since $\mathcal K$ is nonempty, closed and convex, the same is true of
$\mathcal C_h$.

Define the discrete Dirichlet energy
\[
  E_h(v_h):=\frac12\|\nabla v_h\|^2,
  \qquad v_h\in V_h,
\]
and let
\(
  \pi_{1,h}:\mathcal H_h\to V_h,
  \,
  \pi_{1,h}(u_h,w)=u_h.
\)
Set
\(
  \Phi_h:=E_h\circ\pi_{1,h}
\)
and define the restricted energy functional
\begin{equation}
  \Psi_h:=\Phi_h+I_{\mathcal C_h}
  \qquad\text{on }\mathcal H_h.
  \label{eq:continuous_discrete_energy}
\end{equation}
Hence
\(
  D(\Psi_h)=\mathcal C_h
\).
Since $\Phi_h$ is a finite continuous convex functional on $\mathcal H_h$ and
$I_{\mathcal C_h}$ is proper, convex and lower semicontinuous on
$\mathcal H_h$, the functional $\Psi_h$ is proper, convex and
lower semicontinuous.

Define the discrete Laplacian $\Delta_h :V_h\to V_h$ by
\begin{equation}
 (-\Delta_h v_h,z_h)_H=(\nabla v_h,\nabla z_h)_H
 \qquad\forall v_h,z_h\in V_h.
 \label{eq:continuous_discrete_laplacian_operator}
\end{equation}
Then the subdifferential of the discrete Dirichlet part $\Phi_h$
in $\mathcal H_h$ is
\begin{equation}
 \partial_h\Phi_h(u_h,w)
 =
 \left\{
 \left(-\frac1\kappa\Delta_hu_h,0\right)
 \right\},
 \qquad (u_h,w)\in\mathcal H_h.
 \label{eq:continuous_discrete_Dirichlet_subdifferential}
\end{equation}
The adjoint of the restricted gap operator
$\Gapop_h:\mathcal H_h\to\Mmu$ is
\begin{equation}
 \Gapop_h^*\eta
 =
 \left(
   \frac1\kappa\Pi_h^0\Aggop\eta,
   -\eta
 \right),
 \qquad \eta\in\Mmu,
 \label{eq:continuous_discrete_gap_adjoint}
\end{equation}
where $\Pi_h^0:H\to V_h$ denotes the $H$-orthogonal projection.
Since $\Gapop_h$ is surjective, the normal-cone chain rule gives
\begin{equation}
 \partial_h I_{\mathcal C_h}(U_h)
 =
 \Gapop_h^*N_{\mathcal K}(\Gapop_hU_h),
 \qquad U_h\in\mathcal C_h.
 \label{eq:continuous_discrete_state_normal}
\end{equation}
Moreover, $\Phi_h$ is finite and continuous on $\mathcal H_h$.
Hence the subdifferential sum rule gives
\[
 \partial_h\Psi_h
 =
 \partial_h\Phi_h+\partial_h I_{\mathcal C_h}.
\]

\begin{proposition}[Restricted Hilbert subdifferential]
\label{prop:continuous_discrete_subdifferential}
The domain of the subdifferential of $\Psi_h$ in $\mathcal H_h$ is
\( D(\partial_h\Psi_h)=\mathcal C_h \).
For $U_h=(u_h,w)\in\mathcal C_h$, with
$q_h=\Gapop_hU_h$, we have
\begin{equation}
 \partial_h\Psi_h(u_h,w)
 =
 \left\{
 \left(
   \frac1\kappa
   \left[-\Delta_hu_h+\Pi_h^0\Aggop\eta\right],
   -\eta
 \right):
 \eta\in N_{\mathcal K}(q_h)
 \right\}.
 \label{eq:spatialsemi_subdifferential}
\end{equation}
Equivalently, by Lemma~\ref{lem:continuous_pointwise_normal},
the condition on $\eta$ can be written as
\[
 \eta(r)\in N_{K_r}(q_h(r))
 \qquad\text{for $\mu$-a.e. }r\in\Rcal.
\]
\end{proposition}

Analogously, for \(U_h,\, \widehat U_h\in \mathcal C_h\), define the restricted coercivity quantity
\begin{equation*}
  \sigma_h(U_h;\widehat U_h):=
\Psi_h(\widehat U_h)-\Psi_h(U_h)
-\sup_{\Xi_h\in\partial_h\Psi_h(U_h)}
\langle\Xi_h,\widehat U_h-U_h\rangle_{\mathcal H_h}.
\end{equation*}
Using the same argument as for \eqref{eq:continuous_sigma_identity}, we obtain, for every \(\Xi_h\in\partial_h\Psi_h(U_h)\),
\begin{equation}\label{eq:continuous_discrete_sigma_bound}
    \Psi_h(\widehat U_h)-\Psi_h(U_h)-\langle\Xi_h,\widehat U_h-U_h\rangle_{\mathcal H_h}\geq
    \sigma_h(U_h;\widehat U_h)
    =
    \frac12\|\nabla(\widehat u_h-u_h)\|^2.
\end{equation}

Given an initial state $U_h^0=(u_h^0,w_h^0)\in\mathcal C_h$,  let $F(t)$ be as in \eqref{eq:continuous_gradient_flow} and define its orthogonal projection onto $\mathcal H_h$ by
\(
  F_h(t):=
  \left(\frac1\kappa\Pi_h^0 f(t),0\right)
  \in\mathcal H_h.
\) The spatially
semidiscrete problem is then to find $U_h=(u_h,w_h)$  such that
\begin{equation}
 \partial_t U_h+\partial_h\Psi_h(U_h)\ni F_h\quad\text{in }\mathcal H_h
 \quad\text{a.e. in }(0,T),
 \qquad U_h(0)=U_h^0.
 \label{eq:continuous_semidiscrete_gradient}
\end{equation}
Here $w_h$ is not restricted to a finite-dimensional space and its subscript $h$ only
records the coupling to the semidiscrete field $u_h$.
The well-posedness follows from the same arguments as in Theorem~\ref{thm:continuous_wellposedness}.

\begin{proposition}[Semidiscrete well-posedness and regularity]
\label{prop:continuous_semidiscrete_wellposedness}
Suppose Assumption~\ref{ass:continuous_model} holds and let
$U_h^0\in\mathcal C_h$. Then the restricted subgradient flow
\eqref{eq:continuous_semidiscrete_gradient} admits a unique solution satisfying
\begin{equation}
  U_h\in H^1(0,T;\mathcal H_h)\cap C([0,T];\mathcal H_h),
  \qquad
  \Psi_h(U_h)\in W^{1,1}(0,T).
 \label{eq:continuous_semidiscrete_regularity}
\end{equation}
Moreover, for every $t\in[0,T]$,
\begin{equation}
 \frac12\int_0^t\|\partial_t U_h\|_{\mathcal H_h}^2\,ds+
 \Psi_h(U_h(t))
 \leq\Psi_h(U_h^0)
 +\frac1{2\kappa}\|f\|_{L^2(0,t;H)}^2.
 \label{eq:continuous_semidiscrete_stability}
\end{equation}
\end{proposition}

\subsection{Semidiscrete error estimates}
  For comparison with the exact solution, we construct a gap-preserving discrete state by compensating the Ritz projection error of the field in the memory variables. The resulting estimate therefore requires no independent spatial approximation, or additional spatial regularity, of the hysteresis variables.

  Specifically, for a.e. $t\in(0,T)$, define the gap-preserving comparison state by
  \begin{equation}
  U_h^*(t)=(R_hu(t),\widetilde w_h(t)),
  \qquad
  \widetilde w_h(t,r):=w(t,r)+(R_hu(t)-u(t)),
  \label{eq:continuous_spatial_comparison}
  \end{equation}
  where $\widetilde w_h$ is chosen so that the comparison state has the same gap as the exact state, i.e., $\Gapop U = \Gapop U_h^*$.
  Indeed,
  \begin{align}
  &R_hu(t)-\widetilde w_h(t,r)
  =u(t)-w(t,r)=q(t,r)\in K_r.
  \label{eq:continuous_spatial_gap_preservation}
  \end{align}
  Here $\widetilde w_h(t)\in\Mmu$ and consequently $U_h^*(t)\in\mathcal C_h$, because
  \[
  \|\widetilde w_h(t)\|_{\Mmu}
  \leq\|w(t)\|_{\Mmu}
  +S_c^{1/2}\|R_hu(t)-u(t)\|.
  \]
  For convenience, we set \(\Lambda:=\kappa+S_c\) and define the spatial data quantity
  \begin{equation}
    \mathcal M_{\mathrm{sp}}
    :=\|f\|_{L^2(0,T;H)}^2
    +\|\nabla u^0\|^2
    +\|\nabla u_h^0\|^2.
    \label{eq:continuous_spatial_data}
  \end{equation}

\begin{theorem}[Semidiscrete spatial error estimate]
  \label{thm:continuous_semidiscrete_error}
  Under Assumption~\ref{ass:continuous_model}, let $U$ solve \eqref{eq:continuous_gradient_flow} and let $U_h$ solve \eqref{eq:continuous_semidiscrete_gradient}. Then
  \begin{align}
  &\max_{t\in[0,T]}\left[
  \kappa\|u(t)-u_h(t)\|^2+\|w(t)-w_h(t)\|_{\Mmu}^2\right]
  +\int_0^T\|\nabla(u-u_h)\|^2\,dt\notag\\
  &\quad\leq 2\left(\kappa\|u^0-u_h^0\|^2+\|w^0-w_h^0\|_{\Mmu}^2\right)
  +
  Ch^2\mathcal M_{\mathrm{sp}}.
  \label{eq:continuous_semidiscrete_Oh}
  \end{align}
  The constant $C$ depends only on $C_{\rm app}$, $C_{\rm ell}$, $\kappa$ and $S_c$.
\end{theorem}

\begin{proof}
  Apply \eqref{eq:continuous_sigma_bound} with $\widehat U = U_h \in D(\Psi)$ and $\Xi = F-\partial_t U \in \partial\Psi(U)$
  and apply \eqref{eq:continuous_discrete_sigma_bound} with
  $\widehat U_h=U_h^* \in D(\Psi_h)$ and $\Xi_h = F_h - \partial_t U_h \in\partial_h\Psi_h(U_h)$.
  Adding the resulting inequalities and using the orthogonality of the load projection
  \[
  \langle F_h,U_h^* - U_h\rangle_{\Hmu}
  =
  \langle F,U_h^* - U_h\rangle_{\Hmu},
  \qquad U_h^* - U_h\in\mathcal H_h,
  \]
  we obtain, with $e_h=U-U_h$,
  \begin{equation}\label{eq:semidiscrete_error_inequality}
    \begin{aligned}
    \frac12\frac d{dt}\|e_h\|_{\Hmu}^2
    &+\frac12\|\nabla(u-u_h)\|^2
    +\frac12\|\nabla(u_h-R_hu)\|^2\\
    &\leq
    \langle F-\partial_t U_h,U-U_h^*\rangle_{\Hmu}
    -\left(\Psi(U)-\Psi_h(U_h^*)\right).
    \end{aligned}
  \end{equation}
The negative energy-difference term is nonpositive because
\begin{equation*}
  \Psi(U)-\Psi_h(U_h^*) = \frac12\|\nabla u(t)\|^2-
   \frac12\|\nabla R_hu(t)\|^2
   =\frac12\|\nabla(u(t)-R_hu(t))\|^2\geq0.
\end{equation*}
We integrate \eqref{eq:semidiscrete_error_inequality} in time and bound the maximum state error and
the integrated energy error separately.
Adding these bounds gives
\begin{align*}
  &\max_{t\in[0,T]}\left[
  \kappa\|u(t)-u_h(t)\|^2+\|w(t)-w_h(t)\|_{\Mmu}^2\right]
  +\int_0^T\|\nabla(u-u_h)\|^2\,dt\notag\\
  &\quad\leq 2 \left( \kappa\|u^0-u_h^0\|^2+\|w^0-w_h^0\|_{\Mmu}^2\right) +
  4\int_0^T
  \left|
  \left(f-\kappa \partial_t u_h
  -\Aggop \partial_t w_h,
  u-R_hu\right)_H
  \right|dt.
\end{align*}
By \eqref{eq:continuous_Ritz_H2} and the Cauchy--Schwarz inequality, the last integral is bounded by
\[
 C_{\rm app}h^2
 \left(\|f\|_{L^2(0,T;H)}
 +\sqrt\Lambda\|\partial_t U_h\|_{L^2(0,T;\Hmu)}\right)
 \|u\|_{L^2(0,T;H^2)}.
\]
To estimate the remaining factors, \eqref{eq:continuous_aprior} yields
\[
 \|\partial_t U\|_{L^2(0,T;\Hmu)}
 \leq\|\nabla u^0\|+\kappa^{-1/2}\|f\|_{L^2(0,T;H)}.
\]
Combining this bound with \eqref{eq:continuous_natural_H2}, \eqref{eq:continuous_derived_laplacian} and \eqref{eq:continuous_memory_aggregate_bound} shows that $\|u\|_{L^2(0,T;H^2)}$ is controlled by the same data, with constants depending only on $C_{\rm ell}$, $\kappa$, and $S_c$.
The stability estimate \eqref{eq:continuous_semidiscrete_stability} controls $\partial_t U_h$ by $\|\nabla u_h^0\|$ and $\kappa^{-1/2}\|f\|_{L^2}$. Young's inequality gives \eqref{eq:continuous_semidiscrete_Oh}.
\end{proof}

The natural gap-preserving initialization is
\begin{equation}
 u_h^0=R_hu^0,\qquad
 w_h^0(r)=w^0(r)+(R_hu^0-u^0).
 \label{eq:continuous_spatial_initialization}
\end{equation}
It gives
\begin{equation}
 \|U^0-U_h^0\|_{\Hmu}^2
 =\Lambda\|u^0-R_hu^0\|^2
 \leq\Lambda C_{\rm app}^2h^2\|u^0\|_V^2
 \label{eq:continuous_initial_spatial_error}
\end{equation}
and $\|\nabla u_h^0\|\leq\|\nabla u^0\|$. Hence \eqref{eq:continuous_semidiscrete_Oh} proves first-order spatial convergence in the weighted $L^2$ state and field energy norms.

\subsection{Space--time error estimates}

We apply the implicit Euler scheme in Section~\ref{sec:continuous_pi_time} to the restricted subgradient flow \eqref{eq:continuous_semidiscrete_gradient} to derive the space--time discretization.
Given $U_{h\tau}^0\in\mathcal C_h$ and $F_h^n = \left(\kappa^{-1}\Pi_h^0 f^n,0\right) \in\mathcal H_h$, $n=1,\ldots,N$, the space--time discrete scheme seeks $U_{h\tau}^n=(u_{h\tau}^n,w_{h\tau}^n)\in \mathcal H_h$, $n=1,\ldots,N$, such that
\begin{equation}
  \delta_nU_{h\tau}^n+
  \partial_h\Psi_h(U_{h\tau}^n)\ni F_h^n.
  \label{eq:continuous_fully_discrete_gradient}
\end{equation}
The resolvent of $\partial_h\Psi_h$ gives a unique state $U_{h\tau}^n\in\mathcal C_h$ at every time step.
Equivalently, the scheme can be written as follows.
\begin{scheme}[Space--time discrete scheme]
  \label{scheme:space--time}
  Given $(u_{h\tau}^0,w_{h\tau}^0)\in D(\Psi_h)$ and $f^n\in H$, seek $u_{h\tau}^n\in V_h$, $n=1,\ldots,N$, satisfying
  \begin{align}
    & \left(\kappa\delta_nu_{h\tau}^n+
    \Aggop\delta_nw_{h\tau}^n,v_h\right)_H
    +(\nabla u_{h\tau}^n,\nabla v_h)_H
    =(f^n,v_h)_H
    \qquad\forall v_h\in V_h,
    \label{eq:continuous_fully_discrete_balance}\\
    &w_{h\tau}^n(r)=u_{h\tau}^n-P_{K_r}\left(u_{h\tau}^n-w_{h\tau}^{n-1}(r)\right), \quad\text{for $\mu$-a.e. }r\in\Rcal.
    \label{eq:continuous_fully_discrete_play}
  \end{align}
\end{scheme}

The piecewise affine and piecewise constant time reconstructions
$U_{h\tau}$ and $\overline U_{h\tau}$ are defined as in
\eqref{eq:continuous_state_reconstructions}.
Similarly, set $\overline F_h(t)=F_h^n$ on $I_n$, with
$\overline F_h(0)=F_h^0:=(\kappa^{-1}\Pi_h^0f^0,0)$.
Since \eqref{eq:continuous_fully_discrete_gradient} is the implicit Euler
discretization of the restricted subgradient flow
\eqref{eq:continuous_semidiscrete_gradient},
we apply the argument of
Theorem~\ref{thm:continuous_first_order} in $\mathcal H_h$
with energy $\Psi_h$ and load $F_h$.
With the initialization
\eqref{eq:continuous_spatial_initialization},
this yields the following temporal error estimate
uniformly in $h$.

\begin{corollary}[Uniform-in-$h$ temporal error estimate]
\label{cor:continuous_restricted_temporal_error}
Suppose Assumption~\ref{ass:continuous_model} holds,  let $U_h = (u_h,w_h)$ be the solution of the semidiscrete problem \eqref{eq:continuous_semidiscrete_gradient}, and assume that
\[
 \Delta u^0\in H,
 \qquad f\in BV_r(0,T;H).
\]
Initialize $u_{h\tau}^0 = u_h^0$ and $w_{h\tau}^0 = w_h^0$ by \eqref{eq:continuous_spatial_initialization} and set
$f^n=f(t_n)$ for $n=0,\ldots,N$, using the left trace at $T$.
Then the space--time discrete solution $(u_{h\tau},w_{h\tau})$ to Scheme~\ref{scheme:space--time} satisfies
\begin{align}
 &\max\Bigg\{
 \max_{t\in[0,T]}\left(
 \|u_h(t)-u_{h\tau}(t)\|^2
 +\|w_h(t)-w_{h\tau}(t)\|_{\Mmu}^2\right)^{1/2},\notag\\
 &\hspace{3em}\left(\|u_h-u_{h\tau}\|_{L^2(0,T;V)}^2
 +\|u_h-\overline u_{h\tau}\|_{L^2(0,T;V)}^2\right)^{1/2}\Bigg\}\notag\\
 &\quad\leq\frac\tau{\sqrt\kappa}
 \left[\frac1{\sqrt2}\|f(0)+\Delta u^0\|
 +2\operatorname{Var}(f;[0,T])\right].
 \label{eq:continuous_restricted_temporal_error}
\end{align}
\end{corollary}

\begin{proof}
  Since $u_{h\tau}^0-w_{h\tau}^0 = u^0 - w^0 \in \mathcal K$, we have $U_{h\tau}^0\in \mathcal C_h = D(\partial_h\Psi_h)$.
  The subgradient \eqref{eq:spatialsemi_subdifferential} with the selection $\eta=0$ shows that
  \[
  \left(-\frac1\kappa
  \Delta_h u_{h\tau}^0,0\right)
  \in\partial_h\Psi_h(U_{h\tau}^0).
  \]
  Because $\Delta_h u_{h\tau}^0=\Pi_h^0\Delta u^0$, an admissible initial velocity is
  \[
    \delta_0U_{h\tau}^0 := \left(\frac1\kappa\Pi_h^0(f^0+\Delta u^0),0\right)\in F_h^0-\partial_h\Psi_h(U_{h\tau}^0).
  \]
  The contractivity of the $H$-orthogonal projection $\Pi_h^0$ gives the $h$-independent bound
  \begin{equation*}
  \|\delta_0U_{h\tau}^0\|_{\Hmu}
  \leq\frac1{\sqrt\kappa}\|f^0+\Delta u^0\|.
  \end{equation*}
  This provides the \(h\)-uniform initial-velocity bound required by the estimate underlying Theorem~\ref{thm:continuous_first_order}. Together with
  \[
    \operatorname{Var}(\overline F_h;[0,T])
    \leq\kappa^{-1/2}\operatorname{Var}(f;[0,T]),
    \qquad
    \|F_h-\overline F_h\|_{L^1(0,T;\mathcal H_h)}
    \leq\frac{\tau}{\sqrt\kappa}\operatorname{Var}(f;[0,T]),
  \]
  the theorem gives \eqref{eq:continuous_restricted_temporal_error}.
\end{proof}

\begin{theorem}[First-order space--time discrete error estimate]
\label{thm:continuous_full_discretization}
Suppose Assumption~\ref{ass:continuous_model} holds and let $(u,w)$ be the exact solution of Problem~\ref{prob:parabolic_hysteresis_problem}.  Assume in addition that
\[
 \Delta u^0\in H,
 \qquad f\in BV_r(0,T;H).
\]
Set $f^n=f(t_n)$ for $n=0,\ldots,N$, where \(f(T)\) is interpreted as the left trace, and initialize
\[
u_{h\tau}^0 = R_hu^0, \qquad
w_{h\tau}^0 = w^0 + (R_hu^0-u^0).
\]
Then the space--time discrete solution $(u_{h\tau},w_{h\tau})$ to Scheme~\ref{scheme:space--time} satisfies
\begin{align}
 &\max_{t\in[0,T]}\left(
 \|u(t)-u_{h\tau}(t)\|^2
 +\|w(t)-w_{h\tau}(t)\|_{\Mmu}^2\right)^{1/2}
 +\|u-\overline u_{h\tau}\|_{L^2(0,T;V)}
 \notag\\
 &\qquad\leq
 C(h+\tau)
 \left[
 \|u^0\|_V+\|\Delta u^0\|+
 (1+\sqrt T)
 \{\|f(0)\|+\operatorname{Var}(f;[0,T])\}
 \right].
 \label{eq:continuous_full_discretization}
\end{align}
The constant $C$ depends only on $C_{\rm app}$, $C_{\rm ell}$, $\kappa$ and $S_c$.
\end{theorem}

\begin{proof}
Split the error through the spatially semidiscrete solution:
\[
 U-U_{h\tau}=(U-U_h)+(U_h-U_{h\tau}).
\]
Theorem~\ref{thm:continuous_semidiscrete_error} and \eqref{eq:continuous_initial_spatial_error} control the first term by $O(h)$.
Corollary~\ref{cor:continuous_restricted_temporal_error} controls the second by $O(\tau)$, uniformly in $h$. For the energy term, split
\[
 u-\overline u_{h\tau}
 =(u-u_h)+(u_h-\overline u_{h\tau})
\]
and apply Corollary~\ref{cor:continuous_restricted_temporal_error} and the field energy term of Theorem~\ref{thm:continuous_semidiscrete_error}. Finally,
\[
 \|f\|_{L^2(0,T;H)}\leq
 \sqrt T\bigl(\|f(0)\|+\operatorname{Var}(f;[0,T])\bigr),
\]
which gives \eqref{eq:continuous_full_discretization}.
\end{proof}

\section{Conclusions}\label{sec:conclusion}

In this work, we first formulate the play characteristics as closed convex subsets of the spatial Hilbert space, thereby
covering both canonical pointwise models and spatially nonlocal
constraints.
We then construct a convex subgradient-flow structure by introducing a suitable weighted Hilbert metric and energy functional. The resulting formulation makes the dissipative structure explicit and provides a common variational framework for the subsequent discretization analysis.
Within this framework, we prove first-order convergence of the implicit Euler method under bounded-variation forcing and the
condition $\Delta u^0\in H$, without assuming higher-order temporal regularity of the hysteresis variables. This yields the optimal convergence rate despite the hysteresis-induced temporal nonsmoothness. For the spatial discretization, only the diffusive field is restricted to the finite element space, while the hysteresis variables remain in their natural Hilbert space. A gap-preserving comparison construction then yields an \(O(h)\) semidiscrete estimate without any independent spatial approximation of the hysteresis variables. Together, these results yield the space--time error bound \(O(h+\tau)\) under exact spatial integration and constitutive evaluation.

The present analysis concerns the exact-integration formulation. An important extension is to quantify the additional errors introduced by practical implementations, including spatial quadrature and finite approximations of continuous play superpositions. These approximations introduce consistency errors and may also affect the convexity and dissipation structures exploited in the analysis. Establishing corresponding error estimates while retaining the relevant structural properties would provide a natural link between the present theory and practical discretizations.

\appendix

\section{Representative Prandtl--Ishlinskii models of play type}
\label{app:examples}

The following examples show that Assumption~\ref{ass:continuous_model}
covers continuous and finite PI superpositions as well as spatially
nonlocal play constraints.

\begin{example}[Continuous-density PI law]
\label{ex:engineering_density_pi}
Fix $R_{\max}\in(0,\infty]$, a coefficient $a\geq0$, and a Lebesgue-measurable density
\(
 \rho:(0,R_{\max})\to[0,\infty)
\)
satisfying $\rho\in L^1(0,R_{\max})$. For each
$r\in(0,R_{\max})$, let
\[
 K_r:=[-r,r]\subset\mathbb R
\]
and denote by $\Playop_r[u]$ the corresponding scalar play output.
The continuous-density PI law is
\begin{equation}
 \PIop[u](t)
 =a u(t)+\int_0^{R_{\max}}
   \rho(r)\Playop_r[u](t)\,dr.
 \label{eq:engineering_density_pi}
\end{equation}
Models of this form are widely used in the modeling and control of
piezoelectric and other smart-material actuators
\cite{alJanaidehFundamentalsPI2023,zhang_prandtlishlinskii_2023}.

To express \eqref{eq:engineering_density_pi} in the form of
Definition~\ref{def:continuous_pi}, take
\[
 H:=\mathbb R,
 \qquad
 \Rcal:=(0,R_{\max}),
 \qquad
 c(r):=1.
\]
Let $\mathcal L$ denote the Lebesgue $\sigma$-algebra on $\Rcal$ and
define
\[
 \mu_0(A):=\int_A \rho(r)\,dr,
 \qquad A\in\mathcal L.
\]
Let $(\Rcal,\mathcal A,\mu)$ be the completion of
$(\Rcal,\mathcal L,\mu_0)$. Then
\[
 \int_{\Rcal} c(r)w(r)\,d\mu(r)
 =\int_0^{R_{\max}}\rho(r)w(r)\,dr,
\]
so Definition~\ref{def:continuous_pi} gives exactly
\eqref{eq:engineering_density_pi}.

The measure space $(\Rcal,\mathcal A,\mu)$ is complete and finite, and
\[
 S_c
 =\int_{\Rcal} c(r)\,d\mu(r)
 =\mu(\Rcal)
 =\|\rho\|_{L^1(0,R_{\max})}<\infty.
\]
Moreover, each $K_r$ is nonempty, closed, and convex, with
$0\in K_r$, and
\[
 \operatorname{Gr}K
 =\{(r,z)\in(0,R_{\max})\times\mathbb R:|z|\leq r\}.
\]
This set is closed in $(0,R_{\max})\times\mathbb R$ and hence belongs
to
\[
 \mathcal B\bigl((0,R_{\max})\bigr)
 \otimes\mathcal B(\mathbb R)
 \subseteq
 \mathcal A\otimes\mathcal B(\mathbb R).
\]
Thus the constitutive data associated with
\eqref{eq:engineering_density_pi} satisfy
Assumption~\ref{ass:continuous_model}.

\end{example}

\begin{example}[Finite atomic PI superposition]
\label{ex:engineering_atomic_pi}
Fix a positive integer $m$, a coefficient $a_0\geq0$, thresholds
$r_j>0$, and weights $\alpha_j>0$ for $j=1,\ldots,m$. The corresponding finite PI
superposition is
\begin{equation}
 \PIop[v](t)
 =a_0v(t)+\sum_{j=1}^m\alpha_j\Playop_{r_j}[v](t).
 \label{eq:engineering_atomic_pi}
\end{equation}
Finite superpositions of this form are used as classical PI models
in hysteresis compensation for smart-material actuators
\cite{alJanaidehFundamentalsPI2023,zhang_prandtlishlinskii_2023}.

To express \eqref{eq:engineering_atomic_pi} in the form of
Definition~\ref{def:continuous_pi}, take
\[
  H:=\mathbb R,
  \qquad
 \Rcal:=\{1,\ldots,m\},
 \qquad
 \mathcal A:=2^{\Rcal},
 \qquad
 \mu:=\sum_{j=1}^m\alpha_j\delta_j,
\]
where $\delta_j$ denotes the Dirac measure concentrated at $j$, and set
\[
 a:=a_0,
 \qquad
 c(j):=1,
 \qquad
 K_j:=[-r_j,r_j].
\]
Then
\[
 \int_{\Rcal} c(j)w(j)\,d\mu(j)
 =\sum_{j=1}^m\alpha_jw(j),
\]
so Definition~\ref{def:continuous_pi} gives exactly
\eqref{eq:engineering_atomic_pi}.

The finite measure space $(\Rcal,\mathcal A,\mu)$ is complete, and
\[
 S_c
 =\int_{\Rcal} c(j)\,d\mu(j)
 =\sum_{j=1}^m\alpha_j<\infty.
\]
Moreover, each $K_j$ is nonempty, closed and convex, with $0\in K_j$,
and
\[
 \operatorname{Gr}K
 =\bigcup_{j=1}^m
   \bigl(\{j\}\times[-r_j,r_j]\bigr)
 \in\mathcal A\otimes\mathcal B(\mathbb R).
\]
Thus the constitutive data associated with
\eqref{eq:engineering_atomic_pi} satisfy
Assumption~\ref{ass:continuous_model}.

The standard scalar play is recovered by taking
$m=1$, $a_0=0$, and $\alpha_1=1$. More generally, if $H$ is a separable
Hilbert space, $K\subset H$ is a nonempty closed convex set containing
the origin, and $c_0>0$, a single $H$-valued play element with weight
$c_0$ is obtained by taking
\[
 \Rcal:=\{1\},
 \qquad
 \mathcal A:=2^{\Rcal},
 \qquad
 \mu:=\delta_1,
 \qquad
 c(1):=c_0,
 \qquad
 K_1:=K,
 \qquad
 a:=0.
\]
Here $S_c=c_0$ and
\[
 \operatorname{Gr}K=\{1\}\times K
 \in\mathcal A\otimes\mathcal B(H).
\]
Hence the single-play model is the one-atom member of the present
class.
\end{example}

\begin{remark}
  The canonical spatial versions of the scalar laws above are obtained by taking
  $H=L^2(\Omega)$ and replacing the scalar intervals by the
  pointwise constraint sets \eqref{eq:standard_play_constraint_family}, using threshold
  $r_j$ for component $j$ in the finite superposition.
  For the continuous family, graph measurability follows from
  the continuity of
  $(r,z)\mapsto\|(|z|-r)_+\|_{L^2(\Omega)}$.
  For the finite family, the graph is a finite union of
  measurable rectangles.
  The parameter measures and weights remain unchanged.
\end{remark}

\begin{example}[Global $L^2$-ball constraint]
\label{ex:global_l2_ball}
Let $(\Rcal,\mathcal A,\mu)$ and $c$ satisfy the measure and coefficient
conditions in Assumption~\ref{ass:continuous_model}.
Let $H=L^2(\Omega)$ and let $\varrho:\Rcal\to[0,\infty)$ be measurable. Set
\[
 K_r:=\{z\in H:\|z\|_{L^2(\Omega)}\leq \varrho(r)\}.
\]
Then $K_r$ is nonempty, closed and convex, with $0\in K_r$,
and its metric projection is
\[
 P_{K_r}z=
 \begin{cases}
  z, & \|z\|_{L^2(\Omega)}\leq\varrho(r),\\[1mm]
  \displaystyle\frac{\varrho(r)}{\|z\|_{L^2(\Omega)}}z,
     & \|z\|_{L^2(\Omega)}>\varrho(r).
 \end{cases}
\]
For $\varrho(r)>0$ and $\|z\|_{L^2(\Omega)}>\varrho(r)$, the projection at
each spatial point involves the global norm $\|z\|_{L^2(\Omega)}$.
It therefore couples spatially separated parts of the field,
in contrast to the pointwise constraint
\eqref{eq:standard_play_constraint_family}.
Moreover, the map
\[
 (r,z)\longmapsto\|z\|-\varrho(r)
\]
is $\mathcal A\otimes\mathcal B(H)$-measurable. Hence
\[
 \operatorname{Gr}K
 =\{(r,z)\in\Rcal\times H:\|z\|\leq\varrho(r)\}
 \in\mathcal A\otimes\mathcal B(H).
\]
Thus the resulting constitutive data satisfy
Assumption~\ref{ass:continuous_model}.
\end{example}

\section{Measurability of parameter-dependent metric projections}
\label{app:measurable_projection}

The following result supplies the measurable-projection argument used in the proof of Lemma~\ref{lem:continuous_pointwise_normal}.

\begin{proposition}[Measurable metric projection]
\label{prop:measurable_metric_projection}
Let $(\Rcal,\mathcal A,\mu)$ be a complete finite measure space, let $H$ be a separable Hilbert space, and let $K:\Rcal\rightrightarrows H$ be a set-valued map
with nonempty closed convex values and measurable graph
\[
 \operatorname{Gr}K\in\mathcal A\otimes\mathcal B(H).
\]
If $x:\Rcal\to H$ is strongly measurable, then the map
\[
 r\longmapsto P_{K_r}x(r)
\]
is strongly measurable.
\end{proposition}

\begin{proof}
Set
\[
 \phi(r,z):=-\|x(r)-z\|^2,
 \qquad (r,z)\in\Rcal\times H.
\]
We first show that $\phi$ is $\mathcal A\otimes\mathcal B(H)$-measurable.
Since $(\Rcal,\mathcal A,\mu)$ is complete and $H$ is separable, the strong measurability of $x$ implies that $x$ is $\mathcal A/\mathcal B(H)$-measurable \cite[Remark~1.1.18 and Corollary~1.1.10]{hytonen_analysis_2016}.
Define the map
\[
 \Theta:\Rcal\times H\to H\times H,
 \qquad
 \Theta(r,z):=\bigl(x(r),z\bigr).
\]
Both coordinate maps of $\Theta$ are
$\mathcal A\otimes\mathcal B(H)/\mathcal B(H)$-measurable.
Since $H$ is separable,
$\mathcal B(H\times H)=\mathcal B(H)\otimes\mathcal B(H)$.
Hence $\Theta$ is $\bigl(\mathcal A\otimes\mathcal B(H)\bigr)
/\mathcal B(H\times H)$-measurable.
Since $(u,z)\mapsto-\|u-z\|^2$ is continuous on $H\times H$,
the composition
\[
 \phi(r,z)=-\|x(r)-z\|^2
\]
is $\mathcal A\otimes\mathcal B(H)$-measurable.

Because a separable Hilbert space is a Suslin space, Lemma~III.39 and its subsequent application in
\cite{castaingValadierConvexAnalysis1977} apply to
$\phi$ and $K$. They show that
\[
 m(r):=\sup_{z\in K_r}\phi(r,z)
\]
is measurable with respect to the universal completion $\widehat{\mathcal A}$ of $\mathcal A$. Moreover, if the supremum is attained for every $r$, the argmax set-valued map
\[
 \Gamma(r):=\{z\in K_r:\phi(r,z)=m(r)\}
\]
has a $\widehat{\mathcal A}\otimes\mathcal B(H)$-measurable graph and admits a $\widehat{\mathcal A}$-measurable selection.

For every $r$, the set $K_r$ is nonempty, closed and convex. The projection theorem \cite[Theorem~3.16]{bauschke_convex_2017} therefore implies that
the supremum is attained at exactly one point and
\[
 \Gamma(r)=\{P_{K_r}x(r)\}.
\]
Consequently, the $\widehat{\mathcal A}$-measurable selection supplied above is precisely
\(r\mapsto P_{K_r}x(r)\).
Since $\mu$ is finite and $(\Rcal,\mathcal A,\mu)$ is complete, we have
\(
  \mathcal A\subseteq\widehat{\mathcal A}
  \subseteq\mathcal A^{\mu}=\mathcal A,
\)
where $\mathcal A^\mu$ denotes the $\mu$-completion of $\mathcal A$.
Hence $r\mapsto P_{K_r}x(r)$ is
$\mathcal A/\mathcal B(H)$-measurable.
Finally, since $H$ is separable, this map is separably valued, and therefore strongly measurable by \cite[Corollary~1.1.10]{hytonen_analysis_2016}.
\end{proof}

\enlargethispage{\baselineskip}
\bibliographystyle{siamplain}
\bibliography{bib/app,bib/constitution,bib/pde,bib/ref,bib/fem}
\end{document}